\documentclass[12pt,notitlepage,a4paper,reqno,ps]{amsart}
\usepackage{eurosym}
\usepackage{amssymb}
\usepackage{amstext}
\usepackage{color}
\usepackage{enumitem}
\usepackage{esint}
\newcommand{\N}{\mathbb{N}}
\newcommand{\R}{\mathbb{R}}
\newcommand{\M}{\mathcal{M}}
\newcommand{\C}{\mathcal{C}}

\newcommand{\medint}{-\kern  -,375cm\int}

\newcommand{\textmean}[1]{- \hskip-.9em \int_{#1}}

\theoremstyle{plain}
\newtheorem{theorem}{Theorem}[section]
\newtheorem{lemma}[theorem]{Lemma}
\newtheorem{proposition}[theorem]{Proposition}
\newtheorem{corollary}[theorem]{Corollary}

\newtheorem{remark}[theorem]{Remark}
\newtheorem{definition}[theorem]{Definition}

\def\e{\varepsilon}
\numberwithin{equation}{section}
\makeatletter

\renewcommand{\p@enumi}{\thesection.}
\makeatother  \allowdisplaybreaks

\newenvironment{giacomorev}{\color{blue}}{\color{black}}
\newcommand{\bg}{\begin{giacomorev}}
	\newcommand{\eg}{\end{giacomorev}}
\definecolor{vg}{rgb}{0.1, 0.3, 0.15}
\title{Homogenization of non-local integral functionals via two-scale Young measures}

\author{Giacomo Bertazzoni}
\address{ Dipartimento di Matematica e Informatica, Universit\'a degli studi di Ferrara, Via Machiavelli 35, 44121 Ferrara, Italy}
\email{brtgcm@unife.it}
\author{Andrea Torricelli}
\address{Dipartimento di Scienze Matematiche "G. L. Lagrange", Via Duca degli Abruzzi 24, 10129 Torino, Italy}
\email{andrea.torricelli@polito.it}
\author{Elvira Zappale}
\address{Dipartimento di Scienze di Base ed Applicate per l'Ingegneria, Sapienza, Universit\'a di Roma, via Antonio Scarpa 16, 00161 Roma, Italy,  and CIMA, Universidade de \'Evora, Portugal}
\email{elvira.zappale@uniroma1.it}

\begin{document}
	\maketitle

	\begin{abstract}
		We prove a homogenization result in terms of two-scale Young measures for non-local integral functionals. The result is obtained by means of a characterization of two-scale Young measures.
		
		{\scriptsize
			}
	\end{abstract}
	
	\maketitle
	
	\begin{center}
		\fbox{\today}
	\end{center}
	
	%%%%%%%%%%%%%%%%%%%%%%%%%%%%%%%%%%%%%%%%%%%%%%%%%%%%%%%%%
	
	\section{Introduction and Main Results}

	Homogenization has attracted much attention in last five decades, due to the many applications that mixtures and finely heterogeneous media have in applied science and technology, see for instance \cite{MS} and \cite{CEDA} among a wider pioneering literature. 
	In fact both classical and more recent theories of
	ferro-magnetics and mechanics provide an accurate macroscopic description of phenomena occuring in materials which incorporate fine structures at the microscopic level. We refer to \cite{AMMZ, BB1, BF, BMZ, DDI, DEZ, FZ,  FTGNZ, FTN} for applications to the theory of structured deformations, thin structures, non simple materials, evolution problems with nonstandard growth, elastomers, electromagnetic phenomena and constrained models. Here we want to adopt the homogenization techniques to formalize how the effects at micro-level
	emerge at the macro-level in the study of more general phenomena, such as peridynamics, where classical gradient are replaced by non classical one and, consequently non-local integral functions take over the role of standard integrals. 
	
	Indeed, adopting a different view point than \cite{BCPD}, the aim of this paper consists of studying the asymptotic behavior as $\varepsilon \searrow 0^+$ of  $I_\varepsilon:L^p(\Omega;\mathbb R^d) \to \mathbb R^+$ defined as
	\begin{align}\label{functI}
		I_\varepsilon(u):=\iint_{\Omega \times \Omega} W\left(x,x', \frac{x}{\varepsilon}, \frac{x'}{\varepsilon}, u(x), u(x')\right)dx dx'
	\end{align}
	where $\Omega \subset \mathbb R^N$, 
	is an open subset,
	$u:\Omega \to \mathbb R^d$ is in some Lebesgue space $L^p$, $1<p<\infty$
	and the
	integrand $W:\Omega \times \Omega \times\mathbb R^N \times \mathbb R^N \times\mathbb R^d \times \mathbb R^d \to \mathbb [0,+\infty)$
	has some measurability and continuity properties to be specified later
	and $W(\cdot,\cdot,\xi,\eta)$ is  periodic in the first two variables, i.e. 
	we will equivalently denote  $I_\varepsilon$ by 
	\begin{align*}
		I_\varepsilon(u)=\iint_{\Omega \times \Omega}W\left(x, x,'\Big\langle\frac{x}{\varepsilon} \Big\rangle, \Big\langle\frac{x'}{\varepsilon}\Big\rangle, u(x), u(x')\right)dx dx',
	\end{align*}
	where $\langle t \rangle$ denotes the vector in $(0,1)^N$ whose components are the fractional parts of $t \in \mathbb R^N$.

	This kind of non-local functional appears  in the mathematical modeling
	of peridynamics \cite{S}, ferromagnetics \cite{R}, etc.
	For instance, in the first setting, i.e. peridynamics: $\Omega$ represents the reference configuration, $u$ is the deformation and $I_\varepsilon$ is the associated energy.
	
	In general $I_\varepsilon$ is not lower semicontinuous hence the existence of minimizers is not guaranteed, and one should look for its relaxed version, for which we refer to \cite{BMC18},  for instance, in order to detect the microstructure of the material in the context of nonlinear elasticity.
	Here we follow an analogous approach in terms of suitable Young measures,
	extending first
	$L^p$
	to the
	space of Young measures equipped with the narrow topology.\\
	Thus, we can provide a full
	characterization of the $\Gamma$-limit. In fact, the Young measures are so far a crucial tool for obtaining a limiting formula. 
	We refer to \cite{ABM,FLbook, Rindlerbook}, etc.,  for background on Young measures and to \cite{BBS2008, Barchiesi2006, Barchiesi2007, Pedregal2006} for the two-scales Young measures that are a key tool in this homogenization setting.

	As already emphasized in \cite{BMC18, KZSIMA, Pedregal2021, Pedregal2024}, the relaxation and homogenization of \eqref{functI} in
	$L^p$ is a very delicate issue; in fact, it is not clear if the limiting energy  (as $\varepsilon \to 0$) has still the same form.
	
	In our subsequent analysis, we will make use of the two-scale Young measures introduced by \cite{E} and later developped by \cite{V2}, defined in Definition \ref{def-two-scale}, to compute, via $\Gamma$-convergence with respect to a suitable topology, the limit energy of \eqref{functI}, under the assumption on the energy density $W$ of being an admissible integrand (as introduced in \cite{BB}) in the sense of Definition \ref{admisint}.
	We also stress that there is no loss of generality in assuming that $W$ is {\it symmetric}, as observed in \cite{Pedregal2016} and \cite{BMC18}, i.e. 
	\begin{equation}\label{symeq}
		W(x, x_0, 
		y, y_0, \xi, \xi_0)
		= W(x_0
		, x, y_0
		, y, \xi_0,\xi) \end{equation}
	for a.e. $x, x_0 \in  \Omega$, and all $ y, y_0 \in Q,  \xi, \xi_0 \in \mathbb R^d.$

	\begin{theorem}
		\label{hom}
		Let $p > 1$ and assume $W: \Omega \times \Omega \times Q \times Q \times \R^d \times \R^d \rightarrow [0,\infty]$ is a symmetric admissible integrand and there exists $a, \alpha \in L^1(\Omega \times \Omega)$ and $c > 0$ such that
		\begin{equation}
			\label{pgrowth}
			\alpha(x,x') + \frac{1}{c}|\xi|^p \le W(x,x', y, y', \xi, \xi') \le a(x,x') + c(|\xi|^p + |\xi'|^p)
		\end{equation}
		for a.e. $x,x' \in \Omega$, all $y, y' \in \R^d$ and $\xi, \xi' \in \mathbb R^d$. Let $\varepsilon>0$ and let $\{I_\varepsilon\}_\varepsilon$ be the family of functionals in \eqref{functI}. Then, $\{I_\varepsilon\}_\varepsilon$ $\Gamma$-converges, with respect to the weak $L^p(\Omega;\mathbb R^d)$ convergence, to $I_{hom}:L^p(\Omega,\R^d)\to [0,+\infty)$, where $I_{hom}$ is defined as
		\begin{equation*}
			I_{hom}(u):=\min_{\nu \in \mathcal{M}_u} \iint_{\Omega\times\Omega}\iint_{Q\times Q}\iint_{\R^{d}\times\R^{d}} W(x, x', y, y', \xi,\xi')d\nu_{(x,y)}(\xi) d\nu_{(x',y')}(\xi')dydy'dxdx',
		\end{equation*}
		where 
		\begin{align}
			\label{Mu}
			\nonumber \mathcal{M}_u := \Big\{ \nu \in L^{\infty}_w(\Omega \times Q, \mathcal{M}(\R^d)): &\{ \nu_{(x,y)}\}_{(x,y)\in \Omega \times Q}   \text{ is a two-scale Young measure such that} \\
			& \int_Q \int_{\R^d} \xi d\nu_{(x,y)}(\xi)dy = u(x) \Big\}.
		\end{align}
		
	\end{theorem}

	\begin{remark}
		\label{remark1}
		We would like to underline that in the previous theorem we chose the image set $[0,+\infty]$ for semplicity's sake. Indeed, the result still holds if we assume $W$ bounded from below and having values in $\R\cup\{\infty\}$. 
	\end{remark}
	
	In order to prove Theorem \ref{hom}, we characterize two-scale Young measures by proving the following result.
	We will use the strategy introduced in \cite{KP1, KP2} and later adopted in \cite{BBS2008}, (and in \cite{FNTZ} for the Sobolev-Orlicz setting). See also \cite{Barchiesi2006} and \cite{Rindlerbook} where some parts of our results have been proven.

	\begin{theorem}
		\label{mainresult}
		Let $Q :=(0,1)^N$, $\Omega$ be an open and bounded subset of $\R^N$ with Lipschitz boundary, and $\nu\in L_w^\infty(\Omega\times Q, \mathcal{M}(\R^{d}))$ be such that $\nu_{(x,y)}\in\mathcal{P}(\R^{d})$ for a.e. $(x,y)\in \Omega\times Q$. Then, the family $\{\nu_{(x,y)}\}_{(x,y)\in \Omega\times Q}$ is a two-scale Young measure if and only if the following conditions hold
		\begin{itemize}
			\item [i)] there exists $1<p<+\infty$ such that
			\[
			(x,y) \mapsto \int_{\R^{d}} |\xi|^p d\nu_{(x,y)}(\xi)\in L^1(\Omega\times Q);
			\]
			\item[ii)] there exists $u_1\in L^p(\Omega,L^p_{per}(Q))$ such that
			\[\int_{\R^{d}}\xi d\nu_{(x,y)}(\xi)=u_1(x,y),\]
			for a.e. $(x,y)\in \Omega\times Q;$
			\item[iii)] for every $f\in \mathcal{E}_p$
			\[
			\int_Q\int_{\R^{d \times N}} f(y,\xi)d\nu_{(x,y)}(\xi)dy\ge f_{hom}(u(x)) \quad \text{for a.e. }x\in \Omega,
			\]
			where
			\begin{align}\label{Ep}
				{\mathcal E}_p:= &\Big\{f: \overline{Q} \times \mathbb R^d \rightarrow \mathbb R \text{ continuous and such that exists} \\
				&\lim_{|\xi| \rightarrow +\infty} \frac{f(y, \xi)}{1 + |\xi|^p} 
				\text{ uniformly with respect to } y \in \overline{Q}\big\},\label{limunif} 
			\end{align}
			\[
			u(x):=\int_Q u_1(x,y)dy,
			\]
			and for every $\xi \in \mathbb R^d$, 
			$f_{hom}(\xi):= $
			\begin{equation}
				\label{f_hom}
				\lim_{T \rightarrow +\infty} \frac{1}{T^N} \inf \left\{\int_{(0,T)^N} f(\langle x \rangle , \xi + v(x)) dx : v \in L^p((0,T)^N, \R^d), \int_{(0,T)^N} v(x) dx = 0 \right\}.
			\end{equation}
		\end{itemize}
	\end{theorem}
	
	For a generalization of the above characterization when $p=1$, in the constrained case (i.e., $u$ satisfying a suitable PDE, not necessarily {\rm curl}), we refer to \cite{ARD}, leaving to future investigation the study of non-local models as in \eqref{functI} when $W$ satisfies a linear growth condition (i.e. \eqref{pgrowth} for $p=1$).
	To conclude this section, we point out Theorem \ref{hom} holds true in the multi-scale case (see \cite{Barchiesi2006, Pedregal2006} for a discussion on the topic). Indeed, the techniques used in the proof (under appropriate symmetry conditions) can be similarly replicated in the case of multiple scales.

	The paper is organized as follows. Section \ref{Pre} is devoted to notation and preliminaries. In section \ref{Intermediate} we prove some technical results and, in particular, Theorem \ref{mainresult}. Finally, in section \ref{MR} we prove the homogenization result, i.e. Theorem \ref{hom} and we compare our limiting formula with the relaxed energy obtained in \cite{BMC18} in the homogeneous isotropic setting.
	
	\section{Preliminaries}\label{Pre}
	We start by fixing the notation that we will use in the paper. 
	\subsection{Notation}
	
	Throughout this work,  $1<p<+\infty$, $\Omega\subset \mathbb R^N$ is an open bounded set with Lipschitz boundary, $\mathcal{A}(\Omega)$ denotes the family of the open subsets of $\Omega$, and for any $m \in \mathbb N, \mathcal{L}^m$ is the Lebsegue measure in $\R^m$. $Q := (0,1)^N$ is the unit cube in $\R^N$. The symbols $\langle \cdot \rangle$ and $[\cdot]$ stand, respectively, for the fractional and integer part of a number, or a vector componentwise. The Dirac mass at a point $a \in \R^m$ is denoted by $\delta_a$. For every Lebesgue measurable set $A \subset \mathbb R^m$, the symbol $\textmean{A}$ stands for the average $(\mathcal{L}^m(A))^{-1}\int_A$.\\
	Let $U$ be an open subset of $\R$. Then:
	\begin{itemize}
		\item[-] For any linear space $X$, by $X'$ we denote its dual space.
		\item[-] $\mathcal{C}_c(U)$ is the space of the continuous functions $f: U \rightarrow \R$ with compact support.
		\item[-]$\mathcal{C}^\infty_{per}(Q)$ is the space of $Q$-periodic functions on $\R^N.$
		\item[-]  $\mathcal{C}^{\infty}_c(U)$ is the space of the smooth functions $f: U \rightarrow \R$ with compact support.
		\item[-] $\mathcal{C}_0(U)$ is the closure of $\mathcal{C}_c(U)$ with respect to the uniform convergence; it coincides with the space of the continuous functions $f: U \rightarrow \R$ such that, for every $\eta > 0$, there exists a compact set $K_{\eta} \subset U$ with $|f| < \eta$ on $U \setminus K_{\eta}$.
		\item[-] $\mathcal C_b(U;\mathbb R^m)$ is the space of continuous and bounded functions from $U$ to $\mathbb R^m$.
		\item[-] $\mathcal{M}(U)$ is the space of real-valued Radon measures with finite total variation. We recall that, by the Riesz Representation Theorem, $\mathcal{M}(U)$ can be identified with the dual space of $\mathcal{C}_0(U)$ through the duality
		\[ \langle \mu, \phi \rangle = \int_U \phi\,d\mu,\,\,\,\mu \in \mathcal{M}(U),\,\,\phi \in \mathcal{C}_0(U).\]
		\item[-] $\mathcal{P}(U)$ denotes the space of probability measures on $U$, i.e. the space of all $\mu \in \mathcal{M}(U)$ such that $\mu \ge 0$ and $\mu(U) = 1$.
		\item[-] $L^1(\Omega, \mathcal{C}_0(U))$ is the space of maps $\phi: \Omega \rightarrow \mathcal{C}_0(U)$ such that:
		\begin{itemize}
			\item[$(i)$] $\phi$ is strongly measurable, i.e. there exists a sequence of simple functions $s_n: \Omega \rightarrow \mathcal{C}_0(U)$ such that $\|s_n(x) - \phi(x)\|_{\mathcal{C}_0(U)} \rightarrow 0$ for a.e. $x \in \Omega$;
			\item[$(ii)$] $x \mapsto \|\phi(x)\|_{\mathcal{C}_0(U)} \in L^1(\Omega)$.
		\end{itemize}
		\item[-] $L^{\infty}_w(\Omega, \mathcal{M}(U))$ is the space of maps $\nu: E \rightarrow \mathcal{M}(U)$ such that:
		\begin{itemize}
			\item[$(i)$] $\nu$ is weak* measurable, i.e. $x \mapsto \langle \nu_x, \phi \rangle$ is measurable for every $\phi \in \mathcal{C}_0(U)$;
			\item[$(ii)$] $x \mapsto \|\nu_x\|_{\mathcal{M}(U)} \in L^{\infty}(\Omega)$.
			The space $L^{\infty}_w(\Omega, \mathcal{M}(U))$ can be identified with the dual space of  $L^1(\Omega, \mathcal{C}_0(U))$ via the duality
			\[ \langle \mu, \phi \rangle = \int_\Omega \int_U \phi(x, \xi) d \mu_x(\xi)dx, \,\,\, \mu \in L^{\infty}_w(\Omega, \mathcal{M}(U)), \phi \in L^1(\Omega, \mathcal{C}_0(U)),\]
			where $\phi(x,\xi):= \phi(x)(\xi)$ for all $(x, \xi) \in \Omega \times U$.
		\end{itemize}
		\item[-] The space $L^p_{per}(Q)$ stands for the $L^p$-closure of all functions $f \in \mathcal{C}(\R^N)$ which are $Q-$periodic.
		\color{black}
	\end{itemize}
	
	\subsection{Young measures} 
	Young measures and two-scale Young measures (introduced in \cite{E} and later developed by Pedregal in \cite{Pedregal1997} and Valadier in \cite{V2}), are a key tool in our subsequent analysis. Here we recall the definitions and main properties, starting from the theory of classical Young measures, see \cite{V0,V1} and \cite{FLbook, Rindlerbook}. 
	\begin{definition}[Young measures]
		\label{YM}
		Let $\nu \in L^{\infty}_w(\Omega, \mathcal{M}(\R^m))$, and for every $n\in \mathbb N$ let $z_n: \Omega \rightarrow \R^m$ be a measurable function. The family of measures $\{ \nu_x\}_{x \in \Omega}$ is said to be the Young measure generated by the sequence $\{z_n\}$ provided $\nu_x \in \mathcal{P}(\R^m)$ for a.e. $x \in \Omega$ and 
		\[ \delta_{z_n} \overset{*}{\rightharpoonup} \nu \text{ in } L^{\infty}_w(\Omega, \mathcal{M}(\R^m)),\]
		i.e. for all $\psi \in L^1(\Omega, \mathcal{C}_0(\R^m))$
		\[ \lim_{n \rightarrow +\infty} \int_\Omega \psi(x, z_n(x))\,dx = \int_\Omega\int_{\R^m} \psi(x,\xi)\,d\nu_x(\xi)dx.\]
		The family $\{\nu_x\}_{x \in \Omega}$ is said to be a homogeneous Young measure if the map $x \mapsto \nu_x$ is indipendent of $x$. In this case, the family $\{\nu_x\}_{x \in \Omega}$ is identified with a single element $\nu$ of $\mathcal{M}(\R^m)$.
	\end{definition}
	
	A proof of this result can be found in \cite{FLbook}.
	
	\begin{theorem}[Fundamental Theorem on Young measures]
		\label{theoremYoung}
		Let $\{ z_n \}$ be a sequence of measurable functions $z_n: \Omega \rightarrow \R^m$. Then, there exists a subsequence $\{ z_{n_k}\}$ and $\nu \in L^{\infty}_w(\Omega, \mathcal{M}(\R^m))$ with $\nu_x \ge 0$ for a.e. $x \in \Omega$, such that $\delta_{z_{n_k}} \overset{*}{\rightharpoonup} \nu$ in $L^{\infty}_w(\Omega, \mathcal{M}(\R^m))$ and the following properties hold:
		\begin{itemize}
			\item[$(i)$] $ \|\nu_x\|_{\M(\R^m)} = \nu_x(\R^m) \le 1$ for a.e. $x \in \Omega$;
			\item[$(ii)$] if $dist(z_{n_k}, K) \rightarrow  0$ in measure for some closed set $K \subset  \R^m$, then $Supp(\nu_x) \subset K$ for a.e. $x \in \Omega$;
			\item[$(iii)$] $\|\nu_x\|_{\M(\R^m)} = 1$ if and only if there exists a Borel function $g: \R^m \rightarrow [0, +\infty]$ such that
			\[ \lim_{|\xi| \rightarrow +\infty} g(\xi) = +\infty\,\text{ and }\, \sup_{k \in \N} \int_\Omega g(z_{n_k}(x))\, dx < +\infty;\]
			\item[$(iv)$] if $f: \Omega \times \R^m \rightarrow [0, +\infty]$ is a normal integrand, then
			\[ \liminf_{k \rightarrow +\infty} \int_\Omega f(x, z_{n_k}(x))\,dx \ge \int_\Omega\int_{\R^m} f(x,\xi)\,d\nu_x(\xi)\,dx;\]
			\item[$(v)$] if $(iii)$ holds and if $f: \Omega \times \R^m \rightarrow [0, +\infty]$ is a Carathéodory integrand such that the sequence $\{ f(\cdot, z_{n_k})\}$ is equi-integrable then
			\[ \lim_{k \rightarrow \infty} \int_\Omega f(x, z_{n_k}(x))\,dx = \int_\Omega\int_{\R^m} f(x, \xi)\,d\nu_x(\xi)dx.\]
		\end{itemize}
	\end{theorem}
	
	We recall the definition of two-scale Young measure, as introduced by \cite{Pedregal2006} and later developed by \cite{Barchiesi2006, Barchiesi2007, BBS2008} and recently extended to the case $p=1$ (i.e. generalized two-scale Young measures) by \cite{ARD}
	\begin{definition}[Two-scale Young measures]
		\label{def-two-scale}
		Let $\nu \in L^{\infty}_w(\Omega \times Q; \M(\R^d))$. The family $\{ \nu_{(x,y)}\}_{(x,y)\in \Omega \times Q}$ is said to be a two-scale Young measure if $\nu_{(x,y)} \in \mathcal{P}(\R^d)$ for a.e. $(x,y) \in \Omega\times Q$ and if for every sequence $\{ \varepsilon_n\} \rightarrow 0$ there exists a bounded sequence $\{ u_n\}$ in $L^{p}(\Omega, \R^d)$ such that for every $z \in L^1(\Omega)$ and $\varphi \in \C_0(\R^{N\times d} )$,
		\begin{align*}
			\lim_{n \rightarrow \infty}\int_\Omega z(x)\varphi\left( \left\langle \frac{x}{\varepsilon_n}\right\rangle, u_n(x)\right)dx 
			= \int_\Omega\int_Q\int_{\R^d} z(x)\varphi(y,\xi)d\nu_{(x,y)}(\xi)dydx.
		\end{align*}
		Namely, $\{ (\langle \cdot/\varepsilon_n \rangle, u_n)\}$ generates the Young measure $\{ \nu_{(x,y)} \otimes dy\}_{x \in \Omega}$ (see Definition \ref{YM}). Finally, we may call $\{ \nu_{(x,y)}\}_{(x,y) \in \Omega   \times Q}$ the two-scale Young measure associated to the sequence $\{u_n\}$. 
	\end{definition}
	We refer to \cite[Theorem 3.5]{Barchiesi2006}  and \cite[Subsection 2.3]{BBS2008} to compare the two-scale Young measure associated with $\{u_n\}$ (according to the above definition), with the Young measure associated to $\left\{\left(\langle\frac{x}{\varepsilon_n}\rangle, u_n(x)\right)\right\}$, which, in principle, at a naive look, could seem a different element in $L_w^\infty(\Omega; \mathbb T^N\times \mathbb R^d)$, where $\mathbb T^N$ is the torus in $\mathbb R^N$. In fact, for a better understanding of the above notion, it is important to recall the notion of two-scale convergence introduced by \cite{N} and later developed by \cite{Allaire1992} and exploited by many authors. 
	
	\begin{definition}[Two-scale convergence]
		Let $\{u_h\}$ be a sequence in $L^1(\Omega, \R^d)$. The sequence $\{u_h\}$ is said to be two-scale convergent to a function $u = u(x,y) \in L^1( \Omega \times Q, \R^d)$ if, for every component $u^i_h$ of $u_h$ and $u^i$ of $u$, $i=1,\dots, d$. 
		\[ \lim_{h \rightarrow \infty} \int_\Omega   \varphi(x)\phi\left( \frac{x}{\e_h}\right)u^i_h(x)dx = \iint_{\Omega \times Q} \varphi(x)\phi(y)u^i(x,y)dxdy, \]
		for any $\varphi \in \C_c^{\infty}(\Omega)$ and any $\phi \in \mathcal C_{per}^{\infty}(Q)$.  We simply write $u_h \rightsquigarrow u$.
	\end{definition}
	
	We recall \cite[Proposition 2.3]{Pedregal2006}.

	\begin{proposition}
		\label{Pedregal2006-2.3}
		Given a two-scale Young measure $\{\nu_{(x,y)}\}_{(x,y) \in \Omega \times Q}$ as in Definition \ref{def-two-scale}, then the mapping $u_1: \Omega \times Q \rightarrow \R^d$ defined by
		\[ u_1(x,y) = \int_{\R^d} \lambda d\nu_{(x,y)}(\lambda)\]
		is the two-scale limit of the sequence $\{u_n\}$ generating $\{\nu_{(x,y)}\}_{(x,y) \in \Omega \times Q}$.
	\end{proposition}
	
	Now, we also recall the definition for the underlying deformation and we show that it does not depend on the generating sequence but only on the two-scale Young measures.
	\begin{remark}
		\label{remarkud}
		Consider $\{\e_n\}$, $\{u_n\}$ and $\{\nu_{(x,y)}\}_{(x,y) \in \Omega \times Q}$ as in Definition \ref{def-two-scale}. Since $\{u_n\}$ is bounded in $L^p$, then there exists a subsequence $\{u_{n_k}\}$ such that $u_{n_k} \rightsquigarrow u_1$ (see for instance \cite[Theorem 0.1]{Allaire1992}). By \cite[Proposition 3.4]{Barchiesi2006}, $u_{n_k}\rightharpoonup u$ in $L^p$, with
		\begin{equation*}
			u(x)=\int_{Q} u_1(x,y)dy.
		\end{equation*}
		It is important to notice that $u$ is uniquely defined. Indeed, by Proposition \ref{Pedregal2006-2.3} it holds
		\begin{equation*}
			u_1(x,y) = \int_{\R^d} \lambda d\nu_{(x,y)}(\lambda).
		\end{equation*}
		So, if we consider another subsequence $\{u_{n_j}\}$ two-scale converging to some function $v_1$, then since $\nu$ is the same for both $\{\e_{n_k}\}$ and $\{\e_{n_j}\}$, it holds
		\begin{equation*}
			u_1(x,y)=\int_{\R^d} \lambda d\nu_{(x,y)}(\lambda)=v_1(x,y).
		\end{equation*}
		In particular, it follows that
		\begin{equation*}
			u(x)=\int_Q\int_{\R^d} \lambda d\nu_{(x,y)}(\lambda)dy,
		\end{equation*}
		and that $u$ does not depend on the chosen subsequence $\{\e_{n_k}\}$. The function $u$ is called the "underlying deformation" of $\{\nu_{(x,y)}\}_{(x,y) \in \Omega \times Q}$.
	\end{remark}
	
	% \textcolor{red}{Also, we can consider a generating sequence which is preserving boundary values.}
	% \begin{lemma}
	% \label{BBS_2.7}
	%    Given a sequence $\{\e_n\}$ converging to $0$ and $\{u_n\}\subset L^p(\Omega, \R^d)$ such that $u_n\rightharpoonup u$ in $L^p(\Omega, \R^d)$, let us assume that $\{(\langle \frac{\cdot}{\e_n} \rangle, u_n)\}$ generates the Young measure $\{\nu_{(x,y)}\otimes dy\}_{x\in \Omega}$. Then there exists a sequence $\{v_n\}\subset L^p(\Omega, \R^d)$ such that $v_n\rightharpoonup u$ in $L^p(\Omega, \R^d)$, $v_n=u$ on neighborhood of $\partial \Omega$, and $\{(\langle \frac{\cdot}{\e_n} \rangle, v_n)\}$ still generates the Young measure $\{\nu_{(x,y)}\otimes dy\}_{x\in \Omega}$.  
	% \end{lemma}
	
	%\begin{proof}
	%The proof follows from the same argument as \cite[Lemma 2.7]%{BBS2008}. We highlights the %differences for the readers' %convenience. Following their %construction of $w_{n,k}$ we deduce %the equintegrability of $\{w_{n,k}\}$ from \cite[(2.1)]%{BBS2008} using Dunford-Pettis %criterion (instead of using De La %Vallée-Poussin criterion after %\cite[(2.2)]{BBS2008}), and so we %obtain the same limit as %\cite[(2.3)]{BBS2008}. From here, %we can use the same diagonalization %argument and this concludes the %proof.
	%\end{proof}
	
	The proof of the following result is omitted, being easier than the one of \cite[Lemma 3.4]{BBS2008}, since the generating sequences are only in $L^p$ and not in $W^{1,p}$.
	\begin{lemma}
		\label{BBS3.4}
		Let $D$ be a Lebesgue measurable subset of $\Omega$, %with Lipschitz boundary, 
		and consider $\mu$ and $\nu$ in $L^\infty_w(\Omega\times Q, \M (\R^d))$ such that $\{\mu_{(x,y)}\}_{(x,y)\in \Omega\times Q}$ and $\{\nu_{(x,y)}\}_{(x,y)\in \Omega\times Q}$ are two-scale Young measures with the same underlying deformation $u\in L^p(\Omega, \R^d)$. Let
		\begin{equation*}
			\sigma_{(x,y)}:=
			\begin{cases}
				&\mu_{(x,y)} \quad \text{if }(x,y)\in D\times Q, \\
				&\nu_{(x,y)} \quad \text{if }(x,y)\in (\Omega\setminus D)\times Q.
			\end{cases}
		\end{equation*}
		Then $\sigma\in L^\infty_w(\Omega\times Q, \M (\R^d))$ and $\{\sigma_{(x,y)}\}_{(x,y)\in \Omega\times Q}$ is a two-scale Young measure with underlying deformation $u$. 
	\end{lemma}
	%\begin{proof}
	%   The proof follows the same argument as \cite[Lemma 3.4]{BBS2008}.
	%\end{proof}
	
	Next we consider the homogeneous case.
	\begin{definition}[Homogeneous two-scale Young measure]
		Given a two-scale Young measure $\{\nu_{(x,y)}\}_{(x,y)\in \Omega\times Q} \in L_w^\infty(\Omega\times Q, \mathcal{M}(\R^d))$ we say that it is homogeneous if the map $(x,y)\mapsto \nu_{(x,y)}$ is independent of $x$. In this case it can be identified with an element of $L^\infty_w(Q, \mathcal{M}(\R^d))$ and we may write $\{\nu_y\}_{y\in Q}\equiv \{\nu_{(x,y)}\}_{(x,y)\in \Omega\times Q}$.
	\end{definition}
	
	This yields the definition of average for a two-scale Young measure and we can also prove that this is an homogeneous two-scale Young measure.
	\begin{definition}[Average]
		\label{average}
		Given a two-scale Young measure $\{\nu_{(x,y)}\}_{(x,y)\in \Omega\times Q}\in L_w^\infty(\Omega\times Q, \mathcal{M}(\R^d))$, its average is the family $\{\bar\nu_y\}_{y\in Q}$ defined by
		\begin{equation*}
			\langle \bar\nu_y,\phi\rangle:=\fint_{\Omega}\int_{\R^d} \phi(\xi)d\nu_{(x,y)}(\xi)dx, \quad \phi\in C_0(\R^d).
		\end{equation*}
		
	\end{definition}
	
	As observed in \cite{BBS2008}, if $\{\nu_{(x,y)}\}_{(x,y)\in \Omega \times Q}$ is a two-scale Young measure, then it can be seen
	that $\overline{\mu} := \overline{\nu}_y \otimes dy$ is the average of $\{\mu_x \}_{x\in \Omega}$ with $\mu_x := \nu_{(x,y)} \otimes dy$ and $\mu \in L^\infty_w(\Omega;\mathcal M(\mathbb R^N \times \mathbb R^d))$. Thus, $\overline \mu$ is a homogeneous Young measure by \cite[Theorem 7.1, pg. 117]{Pedregalbook}.
	Following the same strategy as \cite[Lemma 2.9]{BBS2008}, we prove that  $\{\overline{\nu}_y \}_{y\in Q}$ is actually a homogeneous two-scale Young measure.
	We consider the case of constant underlying deformation $F \in \mathbb R^d$.
	\begin{lemma}
		\label{BBS_2.9}
		Let $\nu \in L^{\infty}_w(Q \times Q, \M(\R^d))$ be such that $\{ \nu_{(x,y)}\}_{(x,y) \in Q \times Q}$ is a two-scale Young measure with underlying deformation $F$, with $F \in \R^d$. Then $\{ \bar{\nu}_y\}_{y \in Q}$ is a homogeneous two-scale Young measure with the same underlying deformation.
	\end{lemma}
	
	\begin{proof}
		First we note that $\bar\nu\equiv\{\overline \nu_y\}_{y\in Q}\in L^\infty_w(Q, \mathcal{M}(\R^d))$, indeed by Fubini's Theorem and by Definition \ref{average} we have that $y\mapsto\bar\nu_y$ is weakly* measurable. Our aim is to show that for every sequence $\{\e_n\}$ there exists $\{v_n\}\in L^p(Q, \R^d)$ such that the sequence $\left\{\left(\langle\frac{\cdot}{\e_n}\rangle,v_n\right)\right\}$ generates the measure $\bar\mu = \bar\nu_y\otimes dy$ and $v_n\rightharpoonup F$ in $L^p$. Let $\{u_n\}\subset L^p(Q, \R^d)$ be the sequence generating the two-scale Young measure $\nu$ and assume that $u_n\rightharpoonup F$ in $L^p$.  Let $\e_n\to 0$ and define $\rho_n:=\e_n\left[\frac{1}{\sqrt{\e_n}}\right]$. By construction there exists $m_n\in \N$, $a^n_i\in \rho_n\mathbb{Z}\cap Q$, and a measurable set $E_n\subset Q$ such that $\mathcal L^N(E_n)\to 0$ and 
		\begin{equation*}
			Q=\bigcup_{i=1}^{m_n}(a_i^n+\rho_nQ)\cup E_n.
		\end{equation*}
		Let
		\begin{equation*}
			v_n(x):= \begin{cases} u_{\rho_n/\e_n}\left(\frac{x-a^n_i}{\rho_n}\right),\quad \text{ if } x \in a^n_i + \rho_nQ \text{ and } i \in \{1, \dots, m_n\}\\ F \quad \text{ otherwise } \end{cases}
		\end{equation*}
		
		By construction, $v\in L^p(Q,\R^d)$ and $v_n\rightharpoonup F$. Let $z\in \mathcal C_c(Q)$ and $\phi\in \mathcal C_0(\R^N\times \R^d)$, then
		\begin{align*}
			\int_Q z(x)\phi\left(\left\langle \frac{x}{\varepsilon_n} \right\rangle, v_n(x)\right)dx&= \sum_{i=1}^{m_n}\int_{a_i^n+\rho_nQ}z(x)\phi\left(\left\langle \frac{x}{\varepsilon_n} \right\rangle, u_{\rho_n/\varepsilon_n}\left(\frac{x-a_i^n}{\rho_n}\right)\right)dx\\
			& + \int_{E_n} z(x)\phi\left(\left\langle \frac{x}{\varepsilon_n} \right\rangle, F\right)\\
			&=\sum_{i=1}^{m_n}z(a_i^n)\int_{a_i^n+\rho_nQ}\phi\left(\left\langle \frac{x}{\varepsilon_n} \right\rangle, u_{\rho_n/\varepsilon_n}\left(\frac{x-a_i^n}{\rho_n}\right)\right)dx\\
			&+\sum_{i=1}^{m_n}\int_{a_i^n+\rho_nQ}(z(x)-z(a_i^n))\phi\left(\left\langle \frac{x}{\varepsilon_n} \right\rangle, u_{\rho_n/\varepsilon_n}\left(\frac{x-a_i^n}{\rho_n}\right)\right)dx\\
			&  + \int_{E_n} z(x)\phi\left(\left\langle \frac{x}{\varepsilon_n} \right\rangle, F\right).
		\end{align*}
		By the uniform continuity of $z$ and the fact that $\mathcal{L}^N(E_n)\to 0$ it follows that
		\begin{align*}
			\int_Q z(x)\phi\left(\left\langle \frac{x}{\varepsilon_n} \right\rangle, v_n(x)\right)dx=\sum_{i=1}^{m_n}\rho_n^Nz(a_i^n)\int_{Q}\phi\left(\left\langle \frac{\rho_nx+a_i^n}{\varepsilon_n} \right\rangle, u_{\rho_n/\varepsilon_n}\left(x\right)\right)dx + o(1),
		\end{align*}
		i.e. that
		\begin{align*}
			\int_Q z(x)\phi\left(\left\langle \frac{x}{\varepsilon_n} \right\rangle, v_n(x)\right)dx=\sum_{i=1}^{m_n}\rho_n^Nz(a_i^n)\int_{Q}\phi\left(\left\langle \frac{x}{\varepsilon_n/\rho_n} \right\rangle, u_{\rho_n/\varepsilon_n}\left(x\right)\right)dx + o(1),
		\end{align*}
		and so passing to the limit for $n\to \infty$ and by Definition \ref{average} we get
		\begin{align*}
			\lim_{n\to \infty} \int_Q z(x)\phi\left(\left\langle \frac{x}{\varepsilon_n} \right\rangle, v_n(x)\right)dx&= \int_Qz(x)dx\int_Q\int_Q\int_{\R^d}\phi(y,\xi)d\nu_{(x,y)}(\xi)dydx\\
			&=\int_Qz(x)dx\int_Q\int_{\R^d}\phi(y,\xi)d\bar{\nu}_y(\xi)dy\\
			&=\langle\bar{\nu},\phi\rangle\int_Q z(x)dx.
		\end{align*}
		By a density argument it is easy to see that the previous identity holds for every $z\in L^1(Q)$ and so that $\{\bar{\nu}_y\}_{y\in Q}$ is the homogeneous two scale Young measure generated by $\{v_n\}$. 
		
	\end{proof}
	
	Now we introduce our class of functions, according to the notion introduced in \cite{BB}, and later adopted by \cite{BBS2008, Barchiesi2006,V2}.
	
	\begin{definition}[Admissible integrand]
		\label{admisint}
		A function $f: \Omega\times \Omega \times Q\times Q\times \R^d \times \R^d \rightarrow [0,+\infty]$ is said to be an admissible integrand if, for any $\eta > 0$, there exist compact sets $X^\eta\subset \Omega\times \Omega$, $Y^{\eta} \subset Q\times Q$, with $\mathcal{L}^{2N}((\Omega\times \Omega) \setminus X^{\eta}) < \eta$, $\mathcal{L}^{2N}((Q\times Q) \setminus Y^\eta) < \eta$ and such that $f_{|_{X^\eta \times Y^{\eta} \times \R^d \times \R^d} }$ is continuous. 
	\end{definition}
	
	\begin{remark}\label{remBarchiesi}
		\begin{itemize}
			\item[i)] We observe that by \cite[Lemma 4.11]{Barchiesi2006}, if $f$ is an admissible integrand then for any fixed $\varepsilon>0$ the functiona $(x,x',\xi,\xi')\mapsto f(x,x',x/\varepsilon,x'/\varepsilon,\xi,\xi')$ is $\mathcal{L}^{2N}(\Omega\times \Omega)\otimes\mathcal{B}(\R^{2d})$-measurable, where $\mathcal{B}(\R^{2d})$ is the $\sigma$-algebra of Borel subsets of $\R^{2d}$. In particular, the functional $I_\varepsilon$ (defined in \eqref{functI}) is well defined in $L^p(\Omega,\R^d)$.
			\item[ii)] Moreover, we recall that by the Scorza-Dragoni Theorem it is possible to prove that a Carathéodory function is an admissible integrand. 
		\end{itemize}
	\end{remark}
	
	The next result, stated for generic spaces and dimensions, is a two-scale analog of the [Fundamental Theorem of Young Measures [(iv), (v)], and it was proved in \cite[Theorem 2.8]{Barchiesi2007}. 
	We recall it here since it is a crucial tool to prove Theorem \ref{hom}.
	
	\begin{theorem}
		\label{theoremBarchiesi}
		Let $k,m\in\N$, $R\subset \R^k$ be open and bounded, and $S$ be the unit cube in $\R^k$. Let $\{\varepsilon_n\}$ be a vanishing positive sequence. Let $\{u_n\}\subset L^1(R, \R^m)$ be a bounded sequence. Then, there exist non relabeled subsequences and a two-scale Young measure $\{\nu_{(x,y)}\}_{(x,y)\in R \times S}$ such that it holds
		
		\noindent
		\begin{enumerate}
			\item[(i)] if $\mathcal W:R \times  S \times  \R^m \to [0,+\infty]$ is an admissible integrand then
			\begin{align}
				\label{theoremb1}
				\liminf_{n\to +\infty} \int_{R} \mathcal W\left(x, \left \langle 
				\frac{x}{\varepsilon_n}\right\rangle, u_n(x) \right)dx
				\ge \iint_{R\times S }\overline{\mathcal W}(x,y)dxdy, 
			\end{align}
			where
			\begin{equation*}
				\overline{\mathcal W}(x,y):=\int_{\R^M} \mathcal W(x, y, \lambda)d\nu_{(x,y)}(\lambda);
			\end{equation*}
			\item[(ii)] if $\mathcal W:R \times S  \times \R^m \to \R$ is an admissible integrand and\\ $\left\{\mathcal W\left( \cdot, \left\langle\frac{\cdot}{\varepsilon_n} \right \rangle,u_n(\cdot)\right)\right\}$   is equi-integrable for any $x \in R$, then $\mathcal W(x, y, \cdot)$ is $\nu_{(x,y)}$-integrable for a.e. $(x,y)\in R\times S$, the function $\overline{\mathcal W}$ is in $L^1(R\times S)$ and 
			\begin{align}
				\label{theoremb2}
				\lim_{n\to +\infty} \int_{R} \mathcal W\left(x, \left \langle 
				\frac{x}{\varepsilon_n}\right\rangle, u_n(x) \right)dx
				= \iint_{R\times S}\overline{\mathcal W}(x,y)dxdy.
			\end{align}
		\end{enumerate}
	\end{theorem}
	
	\begin{remark}
		\label{remark2}
		We would like to point out that the statement of Theorem $\ref{theoremBarchiesi}$ is slightly different compared to \cite[Theorem 2.8]{Barchiesi2007}. Indeed, in \cite[Theorem 2.8-i)]{Barchiesi2007} the integrand function $\mathcal W$ is assumed to be positive and finite. As far as we understand the proof still works for $\mathcal W$ having values in $[0,+\infty]$ as long as the set where $\mathcal{W}=\infty$ has null measure.
	\end{remark}

	\begin{remark}
		\label{narrow}
		\begin{itemize}
			\item[i)]   We remark that it is possible to look at Young measures from a different perspective than Definition \ref{YM}. In fact, if $\Omega \subset \R^N$ is an open bounded subset and $p>1$, we can say that a Young measure in $\Omega \times \R^m$ is a positive measure $\mu$ in $\Omega \times \mathbb R^m$, such that its push-forward measure $\pi_\Omega \# \mu$, obtained through the projection $\pi_\Omega$ is the $\mathcal L^N$ measure on $\Omega$, i.e.
			for all Borel subsets $B$ of $\Omega$
			\[
			\pi_\Omega\#\mu(B):= \mu(B\times \mathbb R^m)=\mathcal L^N(B).\]
			This set of measures is denoted by $\mathcal Y(\Omega;\mathbb R^m)$. By \cite[Theorem 4.2.4]{ABM}, $\mathcal Y(\Omega;\mathbb R^m)\subset L^\infty_w(\Omega;\mathcal M(\mathbb R^m))$, with the identification $\mu=\{\mu_x\}_{x\in \Omega}\otimes \mathcal L^N$.
			
			\item[ii)]   On the other hand, \cite[Theorem 4.3.1]{ABM} ensures that in $ \mathcal Y(\Omega;\mathbb R^m)$, the weak* convergence in $L^\infty_w(\Omega;\mathcal M(\mathbb R^m))$  of $\{\mu_n\}\subset\mathcal Y(\Omega;\mathbb R^m)$ towards $\mu\in \mathcal Y(\Omega;\mathbb R^m)$ is equivalent to the narrow
			convergence, where the latter is defined as follows
			\[
			\mu_n \overset{nar}{\rightharpoonup} \mu \Longleftrightarrow \lim_n \iint_{\Omega\times \mathbb R^m}\psi(x,\lambda) d \mu_n(x,\lambda)=\iint_{\Omega \times \mathbb R^m} \psi(x,\lambda)\mu(x,\lambda),
			\]
			for every $\psi \in C_b(\Omega;\mathbb R^m)$.
			For more details on the narrow topology we refer to \cite[Section 4.3]{ABM}.

			Equivalently, by the identification $\mu_n=\{(\mu_n)_x\}_{x\in\Omega}$ and $\mu=\{\mu_x\}_{x\in\Omega}$, we say that $\{\mu_n\}$ narrowly converges to $\mu$ if and only if for every $g\in L^1(\Omega)$ and $h\in \mathcal{C}_0(\R^m)$ it holds
			\begin{equation*}
				\lim_{n\to \infty}\int_\Omega g(x)\int_{\R^m} h(y)d(\mu_n)_x(y)dx=\int_\Omega g(x)\int_{\R^m} h(y)d\mu_x(y)dx. 
			\end{equation*}
			
			\item[iii)] Given $p \geq 1$, $\mathcal Y^p(\Omega;\mathbb R^m)$ is the subset of $\mathcal Y(\Omega;\mathbb R^m)$ such that
			\begin{equation}\label{pmom}
				\iint_{\Omega \times \mathbb R^m}
				|y|^p
				d \mu(x, y) <+\infty.
			\end{equation}
			As a consequence of H\"older’s inequality, $\mathcal Y^p (\Omega;\mathbb R^m)\subset \mathcal Y^q(\Omega;\mathbb R^m)$ if $1 \leq q \leq p$.
			
			We recall also that ${\mathcal Y}^p(\Omega;\mathbb R^m)$ is not closed in $\mathcal Y(\Omega;\mathbb R^m)$ 
			under the narrow topology. Nevertheless, given a bounded sequence $\{v_n\} \subset L^p(\Omega;\mathbb R^m)$ there exists $\mu  \in  \mathcal Y(\Omega, \mathbb R^m)$ such that $\{v_n\}$ (up to a subsequence) generates $\mu$; and $\mu \in \mathcal Y^p(\Omega;\mathbb R^m)$.
			Conversely, for any $\mu \in \mathcal Y^p(\Omega;\mathbb R^m)$,
			there exists a bounded sequence $\{v_n\}\subset  L^p(\Omega;\mathbb R^m)$, generating $\mu$ and such that $\{|v_n|^p\}$ is equi-integrable. We refer to \cite{ABM} and \cite{BMC18} for details.
			
			Moreover, considering a bounded sequence $\{v_n\}\subset L^p(\Omega,\R^m)$ and $\nu\in \mathcal{Y}^p(\Omega,\R^m)$, up to identifying each $v_n$ with the Dirac mass $\delta_{v_n}\in \mathcal{Y}^p(\Omega,\R^m)$, by \cite[Theorem 4.3.1]{ABM} we have that $\{v_n\}$ generates $\nu$ (in the sense of Definition \ref{YM}) if and only if $\delta_{v_n}$ narrowly converges to $\nu$. 
			\item[iv)] Following \cite{V2} and \cite{Barchiesi2006}, we can extend the definition of Young measure of finite $p$ moment to parametrized measures defined in $\Omega$ with values in $S$, with $S=\mathbb T^N\times \mathbb R^d$, equipped with the Borel $\sigma$-algebra, i.e. $\mathcal Y^p(\Omega;S)$.
			In this setting the test functions for the narrow convergence can be given by $C_c(\Omega;\mathbb R^d)\otimes C_c(S;\mathbb R^d)$, see \cite[page 149]{V2}. 
			
			These latter observations motivate the comparison of the homogenized functional $I_{hom}$ of Theorem \eqref{hom} with the functional $\overline I_{hom}$ in Corollary \ref{cornar} computed in terms of suitable Young measures in $\mathcal Y^p(\Omega;\mathbb T^N\times \mathbb R^d)$. 
			
			In fact, we also recall that the relaxation of non-local functionals as in \eqref{functI}, but not dependent on $\varepsilon$, has been obtained in \cite{BMC18} in term of narrow convergence in $\mathcal Y^p(\Omega;\mathbb R^d)$. The analogous stand point for the homogenization will be given in Corollary \ref{cornar}, with the use of the above mentioned subset of $\mathcal Y^p(\Omega;\mathbb T^N\times \mathbb R^d)$. 
		\end{itemize}
		
	\end{remark}

	\subsection{$\Gamma$-convergence}
	We recall the definition and main properties of $\Gamma$-convergence. For a deeper overview of this topic, we refer to \cite{DMbook}. 
	The following is an equivalent definition in the metric setting, which is sufficient for our purposes. %characterization of the $\Gamma$-limit.
	\begin{definition}[\cite{DMbook}, Proposition $8.1$]
		Let $(X,d)$ be a metric space and $F_k: X \rightarrow \mathbb R \cup \{+\infty\}, \forall k \in \N$ be a sequence of functionals. Then $\{ F_k\}$ $\Gamma$-converges to $F: X \rightarrow \mathbb R \cup \{+\infty\}$ if and only if 
		\begin{itemize}
			\item[$(i)$] for every $ x \in X$ and every sequence $\{x_k\}$ such that $ d(x_k,x) \to 0$ as $k\to +\infty$, it is
			\[ F(x) \le \liminf_{k \rightarrow \infty} F_k(x_k); \]
			\item[$(ii)$] for every $x \in X$, there exists a sequence $\{x_k\}$ such that $d(x_k,x)\to 0$ as $k\to +\infty$, such that
			\[ F(x) = \lim_{k \rightarrow \infty} F_k(x_k).\]
		\end{itemize}
	\end{definition}
	
	In fact, we write
	\[
	F(x):\inf\{\liminf F_k(x_k): d(x_k,x)\to 0, \hbox{ as }k \to +\infty\}.
	\]
	A fundamental property we want to underline is that the $\Gamma$-limit is lower semi-continuous with respect to the convergence induced by $d$, see \cite[Proposition 6.8]{DMbook}.
	
	We also provide the definition for $\Gamma$-convergence for a family of functionals.
	
	\begin{definition}\label{defGammafamily}
		Let $(X,d)$ be a metric space. For a positive parameter $\varepsilon$, we say that a family $\{ F_{\varepsilon}\}_{\varepsilon}$ of functionals, with $F_\varepsilon: X \rightarrow \mathbb R \cup \{+\infty\}, \Gamma$-converges to $F: X \rightarrow \mathbb R \cup \{+\infty\}$, with respect to the metric $d$ as $\varepsilon \rightarrow 0^+$, if for all vanishing sequences $\{\varepsilon_k\}$,  $\{F_{\varepsilon_k}\} $ $\Gamma$-converges to $F$, when $k \rightarrow \infty$.
	\end{definition}
	We will write, as for the case of sequences
	\[
	F(x):\inf\{\liminf F_\varepsilon(x_\varepsilon): d(x_\varepsilon,x)\to 0, \hbox{ as }\varepsilon \to 0\}.
	\]
	Finally, we state a theorem, whose proof is omitted being very similar to the one of \cite[Theorem 1.1]{CRZ2011} (for related results see also \cite{BFL}).\\
	\begin{theorem}\label{asthm1.1CRZ}
		Let  $\Omega$ be a bounded open set of $\mathbb R^N$, let $f:\Omega\times \mathbb R^d \to \mathbb R$  be a Carath\'eodory integrand
		satisfying
		\begin{itemize}
			\item[(H1)] $f(\cdot ,b)$ is $Q$-periodic, for all $b \in \mathbb R^d$;
			\item[(H2)] there exist $p> 1$ and a positive constant $C$ such that
			\[\frac{1}{C} |b|^p - C \leq f(x,b) \leq C(1+|b|^p),\]
			for a.e. $x \in \Omega$ and for every $b \in \R^d$.
		\end{itemize}
		For every $\varepsilon >0$, consider the family of functionals $F_\varepsilon : L^p( \mathbb R^d)\to (-\infty,+\infty]$,  defined as
		\[ F_\varepsilon (u) :=\int_\Omega f\left(\frac{x}{\varepsilon}, u(x)\right)dx.\]
		Let $\mathcal F$ be the $\Gamma$-limit of $\{F_\varepsilon\}_\varepsilon$ with respect to the weak convergence in $L^p(\Omega;\mathbb R^d)$, i.e. 
		$\mathcal F(u):=\inf\{\liminf F_\varepsilon(u_\varepsilon): u_\varepsilon \rightharpoonup u \hbox{ in }L^p(\Omega;\mathbb R^d)\}$, then
		\[\mathcal F(u)=\int_\Omega f_{hom}(u(x))dx,\]
		where $f_{hom}$ is the function in \eqref{f_hom}.
	\end{theorem}

	%%%%%%%%%%%%%%%%%%%%%%%%%%%%%%%%%%%%
	
	\section{Characterization}\label{Intermediate}
	
	This section is mainly devoted to the proof of Theorem \ref{mainresult} and to the attainment of other related results which will be used in the sequel.\\
	
	\noindent
	We start by recalling the space $\mathcal E_p$ defined by \eqref{Ep} and \eqref{limunif}. 
	Its properties have been presented in \cite[Section 3]{BBS2008}. We briefly recall the main ones. It is a Banach space under the norm
	\[\|f\|_{\mathcal E_p}:=\sup_{\begin{array}{ll}y \in \overline Q, \\
			\xi \in \mathbb R^{d}
	\end{array}}
	\frac{|f (y,\xi)|}{1 + |\xi|^p}.\]
	
	Moreover, ${\mathcal E}^p$ is isomorphic to the space $\mathcal C({\overline Q}\times({\mathbb R}^{d} \cup \{\infty\}))$.

	Before proving Theorem \ref{mainresult} we start commenting on the formulas \eqref{f_hom}.

	Under the assumptions of Theorem \ref{mainresult}, if $f$ belongs to the class \eqref{Ep}, in view of standard relaxation results (see \cite[Proposition 6.11]{DMbook}, \cite[Theorem 6.68 and Remark 6.69]{FLbook}), 
	\begin{align}
		\inf\left\{\liminf_{\varepsilon \to 0}\int_\Omega f\left(\left\langle\frac{x}{\varepsilon}\right\rangle,u_\varepsilon(x)\right)dx: u_\varepsilon \rightharpoonup u \hbox{ in } L^p(\Omega;\mathbb R^d)\right\}= \label{eqco}\\
		\inf\left\{\liminf_{\varepsilon \to 0}\int_\Omega ({\rm co}f)\left(\left\langle\frac{x}{\varepsilon}\right\rangle,u_\varepsilon(x)\right)dx:u_\varepsilon \rightharpoonup u \hbox{ in } L^p(\Omega;\mathbb R^d)\right\}, \nonumber
	\end{align}
	where ${\rm co} f$ stands for the convex envelope of $f$ with respect to the second variable.
	
	Theorem \ref{asthm1.1CRZ} (see the proof of \cite[Theorem 1.1.]{CRZ2011} within a slightly more general context), guarantees that 
	\begin{align*}
		\inf\left\{\liminf_{\varepsilon \to 0}\int_\Omega f\left(\left\langle\frac{x}{\varepsilon}\right\rangle,u_\varepsilon(x)\right)dx:u_\varepsilon \rightharpoonup u \hbox{ in } L^p(\Omega;\mathbb R^d)\right\}=
		\int_\Omega f_{\rm hom}(u(x))dx,
	\end{align*}
	where $f_{\rm hom}$ is as in \eqref{f_hom}.
	On the other hand, \eqref{eqco}
	\begin{align*}
		\inf\left\{\liminf_{\varepsilon \to 0}\int_\Omega f\left(\left\langle\frac{x}{\varepsilon}\right\rangle,u_\varepsilon(x)\right)dx:u_\varepsilon \rightharpoonup u \hbox{ in } L^p(\Omega;\mathbb R^d)\right\}=
		\int_\Omega ({\rm co} f)_{\rm hom}(u(x))dx,
	\end{align*}
	where, for every $\xi \in \mathbb R^d$, $({\rm co}f)_{hom}(\xi)=$
	\begin{equation}
		\label{cof_hom}
		\lim_{T \rightarrow +\infty} \frac{1}{T^N} \inf \left\{\int_{(0,T)^N} ({\rm co}f)(\langle x \rangle , \xi + v(x)) dx : v \in L^p((0,T)^N, \R^d), \int_{(0,T)^N} v(x) dx = 0 \right\}.
	\end{equation}
	Finally arguing as in \cite[Theorem 14.7]{BDF} it results that 
	$$
	({\rm co} f)_{\rm hom}(\xi)=   \inf \left\{\int_{(0,1)^N} ({\rm co}f)(\langle x \rangle , \xi + v(x)) dx : v \in L^p((0,1)^N, \R^d), \int_{(0,1)^N} v(x) dx = 0 \right\}.
	$$
	Hence, by easier arguments than those in \cite[Proposition 6.24]{FLbook}, we can conclude that $f_{\rm hom}= ({\rm co}f)_{\rm hom}$.\\
	
	The next result, whose proof is omitted, as in \cite[Corollary 3.8]{BBS2008}, follows from the fact that conditions \textit{i), ii)}, and \textit{iii)} in Theorem \ref{mainresult} do not depend on the sequence $\{\varepsilon_n\}$.
	
	\begin{corollary}
		\label{corollary}
		Let $\{u_n\}$ be a bounded sequence in $L^p(\Omega,\R^d)$ and assume that there exists a vanishing sequence $\{\varepsilon_n\}$ such that $\{(\langle\frac{\cdot}{\varepsilon_n}\rangle, u_n)\}$ generates the Young measure $\{\nu_{(x,y)}\otimes dy\}_{x\in \Omega}$. Then $\{\nu_{(x,y)}\}_{(x,y)\in\Omega\times Q}$ is a two-scale Young measure.
	\end{corollary}
	%%%%%%%%%%%%%%%%%%%%%%%%%%%%%%%%%%%%%%%%%%%%%%%%%%%%%%%%%%%%%%%%%%%%%%%%%%%%
	We start by addressing the proof of the necessity of Theorem \ref{mainresult}.

	\begin{theorem}\label{neccond}
		Let $\Omega$ be a bounded and open subset of %with Lipschitz boundary 
		$\R^N$. Let $\nu\equiv\{\nu_{(x,y)}\}_{(x,y)\in \Omega \times Q} \in L^{\infty}_w( \Omega \times Q, \M(\R^d))$ be a two-scale Young measure. 
		Then there exist $u_1\in L^p(\Omega; L^p_{per}(Q, \R^d))$ and $u\in L^p(\Omega, \R^d)$ such that
		\begin{equation}
			\label{cond11}
			u_1(x,y)=\int_{\R^d}\xi d\nu_{(x,y)}(\xi),
		\end{equation}
		\begin{equation}\label{3.21}u(x)=\int_Q u_1(x,y)dy \text{ for a.e. }x \in \Omega,\end{equation}
		\begin{equation}
			\label{cond21}
			f_{hom}(u(x)) \le \int_Q\int_{\R^d} f(y,\xi)d\nu_{(x,y)}(\xi)dy,\text{ for a.e. } x \in \Omega, \text{ for  any } f \in \mathcal{E}_p, 
		\end{equation}
		where $f_{hom}(F)$ is defined as in \eqref{f_hom}
		%\begin{equation*}
		%f_{hom}(F) := \lim_{T \rightarrow +\infty} \frac{1}{T^N} \inf \left\{\int_{(0,T)^N)} f( \langle x \rangle , F + v(x)) dx : v \in L^p((0,T)^N, \R^d), \int_{(0,T)^N)} v(x) dx = 0 \right\},
		%\end{equation*} 
		and  
		\begin{equation}
			\label{cond31}
			\int_{\R^d} |\xi|^p d\nu_{(x,y)}(\xi) \in L^1(\Omega \times Q).
		\end{equation}
	\end{theorem}
	
	\begin{proof}
		By definition of two-scale Young measure there exists a sequence $\{u_n\}\subset L^p(\Omega, \R^d)$ which generates $\{\nu_{x,y}\}_{(x,y)\in \Omega \times Q}$,  whose two-scale limit is a function $u_1$ as in the statement, which, in particular, as observed in Proposition \ref{Pedregal2006-2.3}, satisfies \eqref{cond11}.
		Also the weak limit of $\{u_n\}$
		is a function $u \in L^p(\Omega;\mathbb R^d)$ such that $\int_Q u_1(x,y)dy=u(x)$.

		Property \eqref{cond31} follows from \cite[Theorem 3.6 (iv)]{Barchiesi2006} applied to the function $|\cdot|^p$, once we replace the original generating sequence $\{u_n\}$ by a $p$-equi-integrable one, which is always possible, in view of the Decomposition Lemma \cite[Lemma 8.13]{FLbook}.
		%$\{ |u_{n_k}|^p\}$ is equi-integrable, which is always possible by . 
		
		The proof of \eqref{cond21} is based on the same argument as the one for proving \cite[Lemma 3.1-(ii)]{BBS2008}, with the difference that we use an argument entirely similar to \cite[Theorem 1.1]{CRZ2011} instead of \cite[Theorem 14.5]{BDF}. We write the proof for the readers' convenience.\\
		
		\noindent
		Since $f\in\mathcal{E}_p$, then, as observed in \cite[Section 3]{BBS2008} there exists a constant $c$ such that $|f(x,\xi)|\le c(1+|\xi|^p)$ for every $(x,\xi)\in Q\times\R^d$. In order to apply Theorem \ref{asthm1.1CRZ}, the function $f(x,\cdot)$ must also be $p$-coercive. Since this is not the case we introduce an auxiliary function. Fix $M>0$ and consider $f_M(x,\xi):=\max\{-M, f(x,\xi)\}$, with $(x,\xi)\in Q\times\R^d$. Now we fix $\alpha>0$ and we define $f_{M,\alpha}(x,\xi):=f_M(x,\xi)+\alpha|\xi|^p$ for every $(x,\xi)\in Q\times\R^d$.
		By construction
		\begin{equation*}
			\alpha|\xi|^p-M\le|f_{M,\alpha}(x,\xi)|\le (c+\alpha)(1+|\xi|^p), \quad (x,\xi)\in \overline Q\times\R^d.
		\end{equation*}
		By Theorem \ref{asthm1.1CRZ}, it follows that for every $A\in \mathcal{A}(\Omega)$
		\begin{align}
			\label{disug_alphaM1}
			\liminf_{n\to\infty}\int_A f_{M,\alpha}\left(\Big\langle\frac{x}{\e_n}\Big\rangle, u_n\right)dx\ge \int_A (f_{M, \alpha})_{hom}(u)dx\ge\int_A f_{hom}(u)dx,
		\end{align}
		with $f_{hom}$ defined as in \eqref{f_hom}, since $f_{M,\alpha}\geq f$. By construction we have that
		\begin{align}
			\label{disug_alphaM2}
			\liminf_{n\to\infty}\int_A f_{M,\alpha}\left(\Big\langle\frac{x}{\e_n}\Big\rangle, u_n\right)dx\le \liminf_{n\to\infty}\int_A f_{M}\left(\Big\langle\frac{x}{\e_n}\Big\rangle, u_n\right)dx + \alpha\sup_{n\in\N}\int_A |u_n|^p dx.
		\end{align}
		Combinig \eqref{disug_alphaM1} and \eqref{disug_alphaM2} and passing to the limit for $\alpha\to 0$ we get
		\begin{equation}
			\label{disugM1}
			\liminf_{n\to\infty}\int_A f_{M}\left(\Big\langle\frac{x}{\e_n}\Big\rangle, u_n\right)dx \ge \int_A f_{hom}(u)dx.
		\end{equation}
		Now we consider the set
		\begin{equation*}
			A_n^M:=\left\{ x\in A: f\left(\Big\langle\frac{x}{\e_n}\Big\rangle, u_n\right) \le -M \right\}
		\end{equation*}
		and we observe that by Chebyshev's inequality it holds
		\begin{equation*}
			\mathcal{L}^N(A_n^M)\le \frac{c}{M},
		\end{equation*}
		for some constant $c$ independent from $n$ and $M$. By definition of $f_M$ we have that
		\begin{align}
			\label{disugM2}
			\int_A f_{M}\left(\Big\langle\frac{x}{\e_n}\Big\rangle, u_n\right)dx&=-M\mathcal{L}^N(A_n^M)+ \int_{A\setminus A_n^M} f\left(\Big\langle\frac{x}{\e_n}\Big\rangle, u_n\right)dx\nonumber\\
			&\le\int_{A\setminus A_n^M} f\left(\Big\langle\frac{x}{\e_n}\Big\rangle, u_n\right)dx
		\end{align}
		
		Since $\{u_n\}$ converges weakly in $L^p$ then, in view of the Decomposition Lemma (\cite[Lemma 8.13]{FLbook}), up to a subsequence, we can replace $\{u_n\}$ by a $p$-equi-integrable one, still denoted by $\{u_n\}$,   and so, together with the $p$-growth condition of $f(x,\cdot)$, $\{f(\cdot, u_n)\}$ is also equi-itegrable and so
		\begin{equation}
			\label{residuo}
			\int_{A_n^M} f\left(\Big\langle\frac{x}{\e_n}\Big\rangle, u_n\right)dx \to 0 \quad \text{as } M \to +\infty
		\end{equation}
		uniformly with respect to $n$. Combining \eqref{disugM1}, \eqref{disugM2}, and \eqref{residuo} we get
		\begin{equation*}
			\liminf_{n\to +\infty}\int_{A} f\left(\Big\langle\frac{x}{\e_n}\Big\rangle, u_n\right)dx \ge \int_{A} f_{hom}(u)dx.
		\end{equation*}
		Since $\{f(\cdot, u_n)\}$ is equi-integrable, by Theorem \ref{theoremYoung} (v) we get
		\begin{equation*}
			\lim_{n\to \infty} \int_A f\left(\Big\langle\frac{x}{\e_n}\Big\rangle, u_n\right)dx = \int_A \int_Q \int_{\R^d} f(y, \xi)d\nu_{(x,y)}(\xi)dx,
		\end{equation*}
		the proof follows by combining the last two inequalities and using a localization argument.
		%%%%%%%%%%%%%%%%%%%%%%%%%%%%%%%%%%%%%%%%%%%%%%%%%%%%%%%%%%%%%%%%%%%%%%%%%%%%%%%%
	\end{proof}
	
	%%%%%%%%%%%%%%%%%%%%%%%%%%%%%%%%%%%%%%%%%%%%%%%%%%%%%%%%%%
	We start addressing the proof of the sufficiency of Theorem \ref{mainresult}, by introducing some preliminary notions and by proving some intermediate steps.
	\noindent
	For every $F \in \mathbb R^d$ let
	\begin{align}\label{MF} M_F := \Bigg\{\nu \in L^{\infty}_w(Q, \M(\R^d)) : \{ \nu_y\}_{y \in Q} \text{ is a homogeneous two-scale}\\\ \text{ Young measure and } \int_Q \int_{\R^d} \xi d \nu_y(\xi)dy = F\Bigg \}. \nonumber
	\end{align}

	\begin{remark}
		\label{remark3.3BBS}
		Exactly as observed  in \cite[Remark 3.3]{BBS2008}, the set $M_F$ is independent of $\Omega$, i.e. if $\nu \in M_F$ and $\Omega' \subset \R^N$ is another domain, then for all vanishing sequences $\{ \varepsilon_n\}$ there exists a sequence $  \{v_n\} \subset L^p(\Omega, \R^d)$ such that $\{ (\langle \cdot/\varepsilon_n\rangle, v_n)\}$ generates $\nu_y \otimes dy$. Indeed, let $r > 0$ such that $\Omega' \subset r\Omega$. In fact, given an arbitrary  vanishing sequence $\{ \varepsilon_n \}$, define $\delta_n := \varepsilon_n/r$. Then there exists a sequence $\{ u_n\} \subset L^{p}(\Omega, \R^d)$ such that $\{ (\langle \cdot / \delta_n\rangle,  u_n)\}$ generates the homogeneous Young measure $\nu_y \otimes dy$. Define now $v_n(x):= u_n(x/r)$ so that $\{v_n \} \subset L^{p}(r\Omega, \R^d)$ and thus  $\subset L^{p}(\Omega', \R^d)$. By changing  variables it follows that the sequence $\{ (\langle \cdot/ \varepsilon_n\rangle, v_n)\}$ generates the homogeneous Young measure $\nu_y \otimes dy$ as well.
	\end{remark}
	
	%%%%%%%%%%%%%%%%%%%%%%%%%%%%%%%%
	% \begin{lemma}
	% Let $D$ be an open subset of $E$ with Lipschitz boundary \bg(?)\eg, and consider $\mu$ and $\nu \in L^{\infty}_w(E\times I; \M(\R^{1 \times 1}))$ such that $\{ \mu_{(x,y)}\}_{(x,y) \in E \times I}$ and $\{ \nu_{(x,y)}\}_{(x,y) \in E \times I}$ are two-scale Young measures with the same underlying deformation $u \in W^{1,p}(\Omega, \R)$. Let 
	% \[ \sigma_{(x,y)} := \begin{cases} \mu_{(x,y)} \text{ if } (x,y) \in D \times I, \\ \nu_{(x,y)} \text{ if } (x,y) \in (E \setminus D) \times I.\end{cases}\] 
	% Then $\sigma \in L^{\infty}_w(E \times I, \M(\R^{1\times 1}))$ and $\{ \sigma_{(x,y)}\}_{(x,y) \in E \times I}$ is a two-scale Young measure with underlying deformation $u \in L^P(E, \R)$.
	% \end{lemma}
	%%%%%%%%%%%%%%%%%%%%%%%%%%%%%%%%
	
	\begin{lemma}
		\label{lemmf}
		Let $F \in \mathbb R^d$, $M_F$ as in \eqref{MF} and $\mathcal E_p$ be as in \eqref{Ep}, then
		$M_F$ is a convex and weak* closed subset of $(\mathcal{E}_p)'$.
	\end{lemma}
	
	\begin{proof}
		First we show that $M_F$ is a subset of $(\mathcal{E}_p)'$. Fix $\nu\in\M_F$ and identify it with the homogeneous Young measure $\nu_y \otimes dy$. Let $\{v_n\}\subset L^p(\Omega, \R^d)$ be the sequence  such that $\left\{\left(\langle \frac{\cdot}{\varepsilon_n}\rangle,v_n\right)\right\}$ generates $\nu$. 
		
		Up to a subsequence, one can assume that $\{v_n\}$ is $p$-equi-integrable, see \cite[Lemma 8.12]{FLbook}. Consequently, 
		by \cite[Theorem 3.6, iv)]{Barchiesi2006} it holds that
		\begin{equation*}
			K:=\iint_{Q\times\R^d} |\xi|^p d\nu_y(\xi)dy<+\infty.
		\end{equation*}
		By assumption, for a.e. $y\in Q$, the measure $\nu_y$ is a probability measure on $\Omega$, so for every $f\in \mathcal{E}_p$ we have
		\begin{equation*}
			\int_{Q}\int_{\R^d}f(y, \xi)d\nu_y(\xi)dy\le ||f||_{\mathcal{E}_p}\int_{Q}\int_{\R^d}(1+|\xi|^p)d\nu_y(\xi)dy=(1+K)||f||_{\mathcal{E}_p},
		\end{equation*}
		hence $M_F\subset (\mathcal{E}_p)'$. Since $M_F\subset\overline{M_F}$, then in order to show that $M_F$ is weakly* closed it is enough to show that its weak* closure $\overline{M_F}\subset M_F$. Since $\mathcal{E}_p$ is separable then $(\mathcal{E}_p)'$ is metrizable, hence for every $\nu\in \overline{M_F}$ there exists a sequence $\{\nu^n\}\subset M_F$ such that $\nu^n \overset{*}{\rightharpoonup} \nu$ in $(\mathcal{E}_p)'$. For every $n\in\N$, let $\{v^n_k\}\subset L^p(\Omega, \R^d)$ be the sequence generating $\nu^n$. Since the map $(x,\xi)\mapsto \xi$ is an element of $\mathcal{E}_p$, then by the definition of weak* convergence it follows that
		\begin{equation*}
			\lim_{n\to+\infty}\int_Q\int_\Omega \xi d\nu^n_y(\xi)dy=\int_Q\int_\Omega \xi d\nu_y(\xi)dy=F,
		\end{equation*}
		and thus $\nu$ has $F$ as underlying deformation. We have now to show that $\nu$ is an homogeneous two-scale Young measure. For every $z\in L^1(Q)$ and $\phi\in \mathcal C_0(\R^N\times\R^d)$ we have
		\begin{align*}
			&\lim_{k\to +\infty}\lim_{n\to +\infty}\int_Q z(x)\int_{\Omega}\phi\left(\biggl\langle\frac{x}{\e_k}\biggr\rangle, v^n_k(x)\right)dx\\
			& = \lim_{n\to +\infty}\int_Q\int_Q\int_{\R^d} z(x)\phi(y,\xi)d\nu^n_y(\xi)dydx=\int_Q z(x)dx\int_Q\int_{\R^d}\phi(y,\xi)d\nu_y(\xi)dy 
		\end{align*}
		where in the second equality we used the fact that $\mathcal C_0(\R^N \times \R^d)\subset\mathcal{E}_p$. By a diagonalization argument we can find a sequence $\{k_n\}$ such that once defined $u_n:=v_{k_n}^n$, then it holds
		\begin{equation*}
			\lim_{n\to +\infty}\int_Q z(x)\int_{\Omega}\phi\left(\biggl\langle\frac{x}{\e_n}\biggr\rangle, u_n(x)\right)dx=\int_Q z(x)dx\int_Q\int_{\Omega}\phi(y,\xi)d\nu_y(\xi)dy.
		\end{equation*}
		It follows that $\nu$ is a homogeneous two-scale Young measure, which implies that $M_F$ is weakly* closed. Now we have to prove that $M_F$ is a convex set. Fix $\nu, \mu \in M_F$ and $t\in (0,1)$. Consider  $D:= (0,t)\times (0,1)^{N-1}$ and define
		\begin{equation*}
			\sigma_{(x,y)}:=
			\begin{cases}
				&\nu_{y}\quad\text{if }(x,y)\in D\times Q,\\
				&\mu_{y}\quad\text{if }(x,y)\in (Q\setminus D)\times Q.
			\end{cases}
		\end{equation*}
		By Lemma \ref{BBS3.4}, $\{\sigma_{(x,y)}\}_{(x,y)\in Q\times Q}$ is a two-scale Young measure, while by Lemma \ref{BBS_2.9} we have that $\{\overline{\sigma}_y\}_{y\in Q}$ is a homogeneous two-scale Young measure, hence an element of $M_F$. 
		
		Finally, observe that for every $\phi \in L^1(Q, \mathcal C_0(\R^d))$ it holds
		\begin{align}
			&\int_Q\int_{\R^d}\phi(y,\xi)d\overline{\sigma}_y(\xi)dy=\int_Q\int_Q\int_{\R^d} \phi(y,\xi)d\sigma_{(x,y)}(\xi)dydx \nonumber\\
			&=\int_D\int_Q\int_{\R^d} \phi(y,\xi)d\nu_{y}(\xi)dydx+\int_{Q\setminus D}\int_Q\int_{\R^d} \phi(y,\xi)d\mu_{y}(\xi)dydx \label{baricenter}\\
			&=t\int_Q\int_{\R^d} \phi(y,\xi)d\nu_{y}(\xi)dy+(1-t)\int_Q\int_{\R^d} \phi(y,\xi)d\mu_{y}(\xi)dy. \nonumber
		\end{align}
		which means that $\overline{\sigma}=t\nu+(1-t)\mu.$ In order to conclude that $\overline \sigma \in M_F$, we have to show that $\int_Q \int_{\mathbb R^d}\xi d \overline \sigma_y(\xi)dy=F$. In fact, this is the case: it suffices to apply
		\eqref{baricenter} to $\phi(y,\xi)=\xi$, taking into account that $\int_Q\int_{\R^d} \phi(y,\xi)d\mu_{y}(\xi)dy= \int_Q\int_{\R^d} \phi(y,\xi)d\nu_{y}(\xi)dy=F.$ 
	\end{proof}

	Now we show  that conditions i), ii) and iii) of Theorem \ref{mainresult} are sufficient to characterize two-scale
	Young measures. As in \cite{BBS2008} we start from the homogeneous case.
	The non-homogeneous one will be obtained through a suitable approximation of two-scale Young measures by piecewise constant
	ones.

	\begin{lemma}
		\label{suffhom}
		Let $F \in \R$ and $\nu \in L^{\infty}_w(Q, \M(\R^d))$ be such that $\nu_y \in \mathcal{P}(\R^d)$ for a.e. $y \in Q$. Assume that 
		\begin{equation}
			\label{cond1}
			F = \int_Q \int_{\R^d} \xi d\nu_y(\xi)dy,
		\end{equation}
		\begin{equation}
			\label{cond3}
			\int_Q \int_{\R^d} |\xi|^p d\nu_y(\xi)dy < +\infty,
		\end{equation}
		and that
		\begin{equation}
			\label{cond2}
			f_{hom}(F) \le \int_Q\int_{\R^d} f(y,\xi)d\nu_y(\xi)dy,
		\end{equation}
		for every $f \in \mathcal{E}_p$ where $f_{hom}(F)$ is defined in \eqref{f_hom}
		and $\mathcal E_p$ is the space introduced in \eqref{Ep}
		%\begin{equation*}
		%f_{hom}(F) := \lim_{T \rightarrow +\infty} \frac{1}{T^N} \inf \left\{\int_{(0,T)^N)} f( \langle x \rangle , F + v(x)) dx : v \in L^p((0,T)^N, \R^d), \int_{(0,T)^N} v(x) dx = 0 \right\},
		%\end{equation*} 
		
		Then $\{ \nu_y \}_{y \in Q}$ is a homogeneous two-scale Young measure.
	\end{lemma}
	
	\begin{proof}
		Let $F \in \R^d$ and $\nu \in L^{\infty}_w(Q, \mathcal{M}(\R^d)$ be such that $\nu_y\in \mathcal{P}(\R^d)$ for a.e. $y \in Q$, and that \eqref{cond1}, \eqref{cond3}, and \eqref{cond2} hold. We will proceed by contradiction using the Hahn-Banach Separation Theorem. Assume that $\nu$ is not an element of $M_F$. By Lemma \ref{lemmf}, $M_F$ is a convex and weak* closed subset of $(\mathcal{E}_p)'$. Moreover, by \eqref{cond3} and the fact that $\{ \nu_y\}_{y \in Q}$ is a family of probability measures, we get that $\nu \in (\mathcal{E}_p)'$ as well. Since $\nu \notin M_F$, according to Hahn-Banach Separation Theorem, we can separate $\nu$ from $M_F$ i.e. there exists a linear weak* continuous map $L: (\mathcal{E}_p)' \rightarrow \R$ and $\alpha \in \R$ such that $\langle L , \nu \rangle_{(\mathcal{E}_p)'',(\mathcal{E}_p)'} < \alpha $ and  $\langle L , \mu \rangle_{(\mathcal{E}_p)'',(\mathcal{E}_p)'} \ge \alpha $ for all $\mu \in M_F$. Let $f \in \mathcal{E}_p$ be such that
		\begin{align}
			\label{align1}
			\alpha \le \langle L , \mu \rangle_{(\mathcal{E})_p'',(\mathcal{E}_p)'} = \langle \mu , f \rangle_{(\mathcal{E}_p)',\mathcal{E}_p}
			= \int_Q\int_{\R^d} f(y,\xi)d\mu_y(\xi)dy
		\end{align}
		for all $\mu \in M_F$ and 
		\begin{align}
			\label{align2}
			\alpha > \langle L , \nu \rangle_{(\mathcal{E})_p'',(\mathcal{E}_p)'} = \langle \nu , f \rangle_{(\mathcal{E})_p',\mathcal{E}_p} = \int_Q\int_{\R^d} f(y,\xi) d\nu_y(\xi)dy \ge f_{hom}(F).
		\end{align}
		
		Let $f_H$ be defined as
		\begin{equation*}
			f_H(F):= \inf_{\mu \in M_F} \int_Q \int_{\R^d} f(y,\xi) d\mu_y(\xi)dy,\,\,\,F \in \R^d.
		\end{equation*}
		Then by \eqref{align1}, we have that $\alpha \le f_H(F)$. We are going to show that 
		\begin{equation}
			\label{cont}
			f_H(F) \le f_{hom}(F),
		\end{equation}
		which is in contradiction with \eqref{align2} and proves the lemma. \\
		To prove \eqref{cont}, let $T \in \N$ and $\phi \in L^p((0,T)^N; \R^d)$ such that $\int_{(0,T)^N} \phi(x) dx = 0$. Extend $\phi$ to $\R^N$ by $(0,T)^N$-periodicity and consider the sequence 
		\[ \phi_n(x) := F +  \phi\left(\frac{x}{\varepsilon_n}\right),\]
		where $\{ \varepsilon_n\}$ is an arbitrary vanishing sequence. Let $\varphi \in \mathcal{C}_0(\R^N \times \R^d)$ and $z \in L^1(Q)$. Then, since $T \in \N$, the function $F \mapsto \varphi(\langle y \rangle, F + \phi(y))$ is $(0,T)^N$-periodic and according to the Riemann-Lebesgue Lemma, we get that
		
		\begin{align}
			\label{align3}
			\nonumber \lim_{n \rightarrow +\infty} \int_Q z(x) \varphi\left(\left\langle \frac{x}{\varepsilon_n}\right\rangle, \phi_n(x)\right)dx &= \lim_{n \rightarrow +\infty} \int_Q z(x) \varphi\left(\left\langle \frac{x}{\varepsilon_n}\right\rangle, F + \phi\left(\frac{x}{\varepsilon_n}\right)\right)dx\\
			&= \int_Q z(x) dx \fint_{(0,T)^N} \varphi(\langle y \rangle, F + \phi(y))dy.
		\end{align}

		Observe that
		\begin{align}
			\label{align4}
			&\fint_{(0,T)^N} \varphi(\langle y \rangle, F + \phi(y))dy =\frac{1}{T^N} \sum_{a_i \in \mathbb{Z}^N \cap [0,T)^N} \int_{a_i + Q}  \varphi(\langle y \rangle, F + \phi(y))dy \\
			\nonumber &= \frac{1}{T^N} \sum_{a_i \in \mathbb{Z}^N \cap [0,T)^N} \int_{Q}  \varphi(\langle a_i + y \rangle, F + \phi(a_i + y))dy = \frac{1}{T^N} \sum_{a_i \in \mathbb{Z}^N \cap [0,T)^N} \int_{Q}  \varphi(y, F +\phi(a_i + y))dy.
		\end{align}
		
		Thus, from \eqref{align3} and \eqref{align4}, the pair $\{ (\langle\cdot/ \varepsilon_n \rangle ; \phi_n)\}$ generates the homogeneous Young measure
		
		\[ \mu :=  \sum_{a_i \in \mathbb{Z}^N \cap [0,T)^N} \frac{1}{T^N} \delta_{F + \phi(a_i + y)} \otimes dy.\]
		Then,
		\[ \int_Q \int_{\R^d} \xi d\mu_y(\xi)dy = F,\]
		which implies that $\mu \in M_F$. In addition, 
		\[ \fint_{(0, T)^N} f(\langle y \rangle, F + \phi(y))dy = \int_Q \int_{\R^d} f(y, \xi) d\mu_y(\xi) dy,\]
		and then 
		\[ \fint_{(0,T)^N } f(\langle y \rangle, F + \phi(y))dy \ge f_H(F).\]
		As a consequence, taking the infimum over all $ \phi \in L^p((0,T)^N, \R^d)$ such that $\int_{(0,T)^N} \phi(x) dx = 0$ and the limit as $T \rightarrow +\infty $  we get that $f_H(F) \le f_{hom}(F)$ which implies \eqref{cont}.
	\end{proof}

	The following result is an unconstrained version of \cite[Proposition 3.6]{BBS2008}.  
	\begin{lemma}\label{asProp3.6BBS}
		Let $\nu \in L^\infty_w(\Omega \times Q;\mathcal M(\mathbb R^d))$ be such that the family
		$\{\nu_{(x,y)}\}_{(x,y)\in \Omega \times Q}$ is a two-scale Young measure. Then for a.e. $a \in \Omega$,
		$\{\nu_{(a,y)}\}_{y\in Q}$ is a homogeneous two-scale Young measure.
	\end{lemma}
	\begin{proof}[Proof]
		
		Since $\{\nu_{(x,y)}\}_{(x,y)\in \Omega \times Q}$ is a two-scale Young measure, it satisfies \eqref{cond11},\eqref{cond21} and \eqref{cond31}   Since $u_1(x, \cdot)$ is $Q$-periodic
		for a.e. $x \in \Omega$, integrating \eqref{cond11} %with respect to $y \in Q$, 
		we get \eqref{3.21}.
		%\\int_Q \xi d \nu_{(x,y)}dy = u(x)\]
		%for a.e. $x \in \Omega$ (3.21)
		Furthermore, \eqref{cond31} implies that
		\begin{equation}\label{3.22}\int_Q \int_{\mathbb R^d}|\xi|^p d{\nu_{(x,y)}}dy <+\infty \hbox{ for a.e. }x \in \Omega.
		\end{equation}
		
		Let $E \subset \Omega$  be a set of Lebesgue measure zero such that \eqref{3.21}, \eqref{cond21} and \eqref{3.22} do
		not hold. Then for every $a \in \Omega \setminus E$,
		\[\int_Q \int_{\mathbb R^d} \xi d\nu_{(a,y)}(\xi)d y= u(a),\]
		\[\int_Q \int_{\mathbb R^d}|\xi|^p d \nu_{(a,y)}(\xi)d y < +\infty,\]
		and
		\[\int_Q \int_{\mathbb R^d}f(y,\xi)d \nu_{(a,y)}(\xi)dy \geq f_{hom}(u(a))\]
		for every $f \in \mathcal E_p$.
		As a consequence of Lemma \ref{suffhom}, for every $a \in \Omega \setminus E$, the family $\{\nu_{(a,y)}\}_{y \in Q}$ is a
		homogeneous two-scale  Young measure.
	\end{proof}

	\begin{theorem}\label{necThmmain}
		Let $\Omega$ be a bounded and open subset of $\R^d$. Let $\nu \in L^{\infty}_w( \Omega \times Q, \M(\R^d))$ be such that $\nu_{(x,y)} \in \mathcal{P}(\R^d)$ for a. e. $(x,y) \in \Omega \times Q$. Assume that i)-iii) of Theorem \ref{mainresult} hold, i.e. \begin{equation}
			\label{cond30}
			\int_{\R^d} |\xi|^p d\nu_{(x,y)}(\xi)  \in L^1(\Omega \times Q), %+\infty,\,\,\,\,\text{for a.e. } x \in \Omega,
		\end{equation}
		there exist $u_1\in L^p(\Omega;( L^p_{per}(Q;\mathbb R^d)))$ and $u\in L^p(\Omega;\mathbb R^d)$  such that \begin{equation}
			\label{cond10}
			u_1(x,y)=\int_{\R^d}\xi d\nu_{(x,y)}(\xi),
		\end{equation}
		with $\int_Q u_1(x,y)dy=u(x)$. Assume also that
		\begin{equation}
			\label{cond20}
			f_{hom}(u(x)) \le \int_Q\int_{\R^d} f(y,\xi)d\nu_{(x,y)}(\xi)dy,\text{ for a.e. } x \in \Omega, \text{ for  any } f \in \mathcal{E}_p, 
		\end{equation}
		where $f_{hom}(F)$ is the function in \eqref{f_hom}
		%\begin{equation*}
		%f_{hom}(F) := \lim_{T \rightarrow +\infty} \frac{1}{T^N} \inf \left\{\int_{(0,T)^N} f( \langle x \rangle, F + v(x)) dx : v \in L^p((0,T)^N,\R^d), \int_{(0,T)^N} v(x) dx = 0 \right\},
		%\end{equation*} and that 
		Then $\{ \nu_{(x,y)} \}_{(x,y) \in \Omega \times Q}$ is a  two-scale Young measure with underlying deformation $u$.
	\end{theorem}
	
	\begin{proof}
		%Firstly, we deal with the case where the underlying deformation is zero, afterwards we treat the general case.\\
		%Assume u = 0 and 
		Let $u=0$ a.e. on $\Omega$ and Let $(\varphi, z)$ be in a countable dense subset of $\mathcal{C}_0(\R^N \times \R^d) \times L^1(\Omega)$. Define
		\[ \bar{\varphi}(x):= \int_Q\int_{\R^{d}} \varphi(y, \xi)d\nu_{(x,y)}(\xi)dy.\]
		Since the average with respect to $y$ of $u_1(x, y)$ is $u(x)$ for a.e. $x \in \Omega$, integrating \eqref{cond10} with respect to $y \in Q$, it follows that
		\begin{equation}
			\label{cond1*}
			\int_Q\int_{\R^{d}} \xi d\nu_{(x,y)}(\xi)dy = 0,\,\,\,\,\text{for a.e.  }x \in \Omega.
		\end{equation}
		Let $E \subset \Omega$ be the null Lebesgue measure set such that \eqref{cond1*}, \eqref{cond20} or \eqref{cond30} do not hold. Then for every $a \in \Omega \setminus E$
		\[ \int_Q\int_{\R^{d}} \xi d\nu_{(a,y)}(\xi)dy =  0,\]
		\[ f_{hom}(0) \le \int_Q\int_{\R^d} f(y,\xi)d\nu_{(a,y)}(\xi)dy\]
		for every $f \in \mathcal{E}_p$,
		and
		\[\int_Q \int_{\R^d} |\xi|^p d\nu_{(a,y)}(\xi)dy < +\infty.\]
		
		Now, let $k \in \N$. According to \cite[Lemma 7.9, pg. 129]{Pedregalbook}, there exist points $a^k_i \in \Omega \setminus E$ and positive numbers $\rho^k_i \le 1/k$ such that $a^i_k + \rho^k_i \Omega$ are pairwise disjoint for each $k$, 
		\[ \bar{\Omega} = \bigcup_{i \ge 1} (a^k_i + \rho^k_i\bar{\Omega}) \cup E_k, \,\text{ with } \mathcal{L}(E_k)=0\]
		and 
		\begin{equation}
			\label{eq1.1}
			\int_{\Omega} z(x)\bar{\varphi}(x)dx = \lim_{k \rightarrow +\infty} \sum_{i \ge 1}\bar{\varphi}(a^k_i) \int_{a^k_i + \rho^k_i \Omega} z(x)dx.
		\end{equation}
		For each $k \in \N$, let $m_k \in \N$ be large enough so that
		\begin{equation}
			\label{eq2}
			\left| \sum_{i = 1}^{m_k} \bar{\varphi}(a^k_i)\int_{a^k_i+\rho^k_i \Omega} z(x)dx - \sum_{i \ge 1}\bar{\varphi}(a^k_i)\int_{a^k_i + \rho^k_i \Omega} z(x)dx\right| < \frac{1}{k}.
		\end{equation}
		For fixed $i$ and $k$, by choice of $a^k_i$ and by Lemma \ref{asProp3.6BBS}, %\ref{suffhom} 
		the family $\{\nu_{(a^k_i, y)}\}_{y \in Q}$ is a homogeneous two-scale Young measure. Hence, by Remark \ref{remark3.3BBS}, for every vanishing sequence $\{\varepsilon_n\}$, there exist sequences $\{ u^{i,k}_n\} \subset L^p(a^k_i + \rho^k_i \Omega, \R^d)$ such that
		\[ \lim_{n \rightarrow + \infty} \int_{a^k_i + \rho^k_i \Omega} z(x)\varphi\left( \left\langle \frac{x}{\varepsilon_n}\right\rangle,  u^{i,k}_n(x)\right) dx = \bar{\varphi}(a^k_i) \int_{a^k_i + \rho^k_i \Omega } z(x)dx.\]
		Summing up
		\begin{align*}
			\lim_{n \rightarrow +\infty} \sum_{i = 1}^{m_k} \int_{a^k_i + \rho^k_i \Omega} z(x) \varphi\left( \left\langle \frac{x}{\varepsilon_n} \right\rangle, u^{i,k}_n\right)dx = \sum_{i = 1}^{m_k} \bar{\varphi}(a^k_i)\int_{a^k_i + \rho^k_i \Omega}z(x)dx.
		\end{align*}
		Let us define
		\[ u^k_n(x) := \begin{cases}
			u^{i,k}_n(x) \,\,\,\,\text{ if } x \in a^k_i + \rho^k_i \Omega, \\ 0 \,\,\,\, \text{ otherwise }
		\end{cases}\]
		and remark that $u^k_n \in L^p(\Omega, \R^d)$.  
		Since the sets $a^k_i + \rho^k_i \Omega$ are pairwise disjoint for each $k$ we have that 
		\begin{align*}
			\int_{\Omega} z(x)\varphi \left( \left \langle \frac{x}{\varepsilon_n}\right\rangle,  u^k_n(x) \right)dx   &= \sum_{i \ge 1} \int_{a^k_i + \rho^k_i\Omega} z(x)\varphi\left( \left\langle \frac{x}{\varepsilon_n}\right\rangle,  u^{i,k}_n(x)\right) dx \\  = \sum_{i = 1}^{m_k} \int_{a^k_i + \rho^k_i \Omega} z(x)\varphi\left(\left\langle \frac{x}{\varepsilon_n}\right\rangle,  u^{i,k}_n(x) \right)dx &+ \int_{\Omega \cap \bigcup_{i > m_k}(a^k_i + \rho^k_i \Omega)} z(x) \varphi\left(\left\langle \frac{x}{\varepsilon_n}, u^k_n(x)\right\rangle\right) dx.
		\end{align*}
		But as $z \in L^1(\Omega)$ and $\mathcal{L}^N(\Omega \cap \bigcup_{i > m_k}(a^k_i + \rho^k_i\Omega)) \rightarrow 0$, as $k \rightarrow +\infty$, it follows that
		\begin{equation}
			\label{eq5}
			\lim_{k \rightarrow +\infty} \lim_{n \rightarrow +\infty} \left | \int_{\Omega \cap \bigcup_{i > m_k}(a^k_i + \rho^k_i \Omega)} z(x)\varphi\left( \left\langle \frac{x}{\varepsilon_n}\right\rangle,  u^k_n(x)\right) dx \right| = 0.
		\end{equation}
		Then, gathering (\eqref{eq1.1} - \eqref{eq5}) we obtain that
		\begin{align*}
			&\lim_{k \rightarrow + \infty}\lim_{n \rightarrow + \infty} \int_{\Omega} z(x) \varphi\left( \left\langle \frac{x}{\varepsilon_n}\right\rangle,  u^k_n(x)\right) dx \\& = \lim_{k \rightarrow +\infty} \lim_{n \rightarrow +\infty} \sum_{i = 1}^{m_k} \int_{a^k_i + \rho^k_i \Omega} z(x)\varphi\left( \left\langle \frac{x}{\varepsilon_n}\right\rangle,  u^k_n(x)\right) dx \\ &= \lim_{k \rightarrow + \infty} \sum_{i=1}^{m_k}\bar{\varphi}(a^k_i) \int_{a^k_i + \rho^k_i \Omega} z(x)dx \\ &= \lim_{k \rightarrow +\infty} \sum_{i \ge 1}\bar{\varphi}(a^k_i) \int_{a^k_i + \rho^k_i \Omega} z(x)dx \\ &= \int_{\Omega} z(x) \bar{\varphi}(x) dx.
		\end{align*}
		
		A diagonalization argument implies the existence of a diverging sequence $\{ k_n\} $, as $n \rightarrow \infty$, such that upon setting $u_n := u^{k_n}_n$, then 
		\[ \lim_{n \rightarrow +\infty} \int_{\Omega} z(x) \varphi\left( \left\langle \frac{x}{\varepsilon_n}\right\rangle,  u_n(x)\right) dx = \int_{\Omega} z(x) \bar{\varphi}(x)dx\]
		and $u_n \rightharpoonup 0 $ in $L^p(\Omega, \R^d)$, which completes the proof whenever $u=0$.\\
		Consider now a general $u \in L^p(\Omega, \R^d)$ and $\nu$ satisfying properties \eqref{cond30}-\eqref{cond20}. We define $\tilde{\nu} \in L^{\infty}_w(\Omega \times Q, \mathcal{M}(\R^d))$ by
		\begin{equation}
			\label{nutilde}
			\langle \tilde{\nu}, \psi \rangle := \int_\Omega \int_Q \int_{\R^d} \psi(x, y, \xi -  u(x)) d\nu_{(x,y)}(\xi)dydx,
		\end{equation}
		for every $\psi \in L^1(\Omega \times Q, \mathcal{C}_0(\R^d))$.
		We can easily check that $\tilde{\nu}$ satisfies the analog properties of \eqref{cond30}-\eqref{cond20} with $u=0$.  Hence, applying the first step of the proof, for vanishing every sequence $\{\varepsilon_n\}$ we may find a sequence $\{ \tilde{u}_n\} \subset L^p(\Omega, \R^d)$ such that $\{ (\langle \cdot / \varepsilon_n\rangle , {\tilde u}_n)\}$ generates the Young measure $\{ \tilde{\nu}_{(x,y)} \otimes dy\}_{x \in \Omega}$. Define $u_n:= \tilde u_n +u$. It is easily seen that $\{ (\langle \cdot / \varepsilon_n\rangle , u_n)\}$ generates the Young measure $\{ \nu_{(x,y)} \otimes dy\}_{x \in \Omega}$. Indeed, let $\psi \in L^1(\Omega, \mathcal{C}_0(\R^N \times \R^d))$ and define $\tilde{\psi}(x,y,\xi) := \psi(x,y,\xi +  u(x))$ where $\tilde{\psi} \in L^1(\Omega, \mathcal{C}_0(\R^N \times \R^d))$ as well. Then, by \eqref{nutilde}
		\begin{align*}
			& \lim_{n \rightarrow +\infty} \int_{\Omega} \psi\left( x, \left\langle \frac{x}{\varepsilon_n}\right\rangle,  u_n(x) \right) dx = \lim_{n \rightarrow +\infty} \int_{\Omega} \tilde{\psi}\left( x, \left\langle \frac{x}{\varepsilon_n}\right\rangle,  \tilde{u}_n(x) \right) dx \\ &= \int_{\Omega} \int_Q \int_{\R^d} \tilde{\psi}(x,y,\xi) d\tilde{\nu}_{(x,y)}(\xi)dydx = \int_{\Omega} \int_Q \int_{\R^d} \psi(x,y,\xi) d\nu_{(x,y)}(\xi)dydx 
		\end{align*}
		which completes the proof.
	\end{proof}

	%%%%%%%%%%%%%%%%%%%%%%%%%%%%%%%%%%%%%%%%%%%%%%%%%%%%
	
	\section{Homogenization}\label{MR}

	This section is devoted to the proof of Theorem \ref{hom}. To this end we start by showing  that two-scale Young measures behave as measure products when generated by couples of sequences. This fact is a fundamental tool for achieving our result, and it represents an extension of the pioneering \cite[Proposition 2.3]{Pedregal1997}.

	\begin{proposition}
		\label{Pedregal1997}
		Let $\Lambda = \{ \Lambda_{(x, x, y, y')}\}_{(x, x', y, y')\in\Omega^2\times Q^2} \subset L^\infty_w(\Omega^2 \times Q^2;\mathcal M(\mathbb R^d \times \mathbb R^d))$  be a family of probability measures supported on $\R^d \times \R^d$. $\Lambda$ is the two-scale Young measure associated with a sequence $\{(u_n(x), u_n(x'))\}$, with $\{ u_n \}\subset L^{p}(\Omega, \R^d)$ if and only if $\Lambda_{(x, x', y, y')} = \nu_{(x,y)} \otimes \nu_{(x',y')}$, where $\nu:=\{\nu_{(x,y)}\}_{(x,y)\in \Omega \times Q}\subset L^\infty_w(\Omega \times Q;\mathcal M(\mathbb R^d))$ is the two-scale Young measure associated (in the sense of  Definition \ref{def-two-scale}) with the sequence $\{ u_n\}$.  
	\end{proposition}
	
	\begin{proof}
		Let $\{\e_n\}$ be any vanishing sequence and let $\Lambda$ be the two-scale Young measure associated to the sequence $\{( u_n(x), u_n(x'))\}$ for some bounded sequence $\{ u_n \} \subset L^{p}(\Omega, \R^d)$. We test $\Lambda$ against functions $\theta(x,x')\varphi(y,\xi,y',\xi')$ with $\theta\in L^1(\Omega\times\Omega)$ and $\varphi\in C_0(\R^N\times\R^d\times\R^N\times\R^d)$, (in view of the density of linear combinations of such functions in $L^1(\Omega\times \Omega; C_0(\mathbb R^N \times \mathbb R^d\times \mathbb R^N\times \mathbb R^d)$). In particular, we choose $\varphi$  such that $\varphi(y,\xi,y',\xi')=\varphi_1(y,\xi)\varphi_2(y',\xi')$ and $\theta(x,x')=\theta_1(x)\theta_2(x')$.
		
		\begin{align*}
			&\iint_{\Omega\times\Omega}\theta_1(x)\theta_2(x')\left(\iint_{Q\times Q}\iint_{\R^d\times\R^d} \varphi_1(y', \xi')\varphi_2(y', \xi') d\Lambda_{(x, x', y, y')}(\xi,\xi')dydy'
			\right)dxdx'\\
			&=\lim_{n\to +\infty}\iint_{\Omega\times\Omega}\theta(x,x') \varphi\left(\left\langle \frac{x}{\e_n} \right\rangle, \left\langle \frac{x'}{\e_n} \right\rangle, u_n(x), u_n(x') \right) dxdx'\\
			& =\lim_{n\to +\infty}\iint_{\Omega\times\Omega}\theta_1(x)\theta_2(x') \varphi_1\left(\left\langle \frac{x}{\e_n} \right\rangle, u_n(x) \right)\varphi_2\left(\left\langle \frac{x'}{\e_n} \right\rangle, u_n(x') \right) dxdx'\\
			&=\lim_{n\to +\infty} \left( \int_\Omega \theta_1(x)\varphi_1\left(\left\langle \frac{x}{\e_n} \right\rangle, u_n(x)\right)dx \right) \left( \int_\Omega \theta_2(x')\varphi_2\left(\left\langle \frac{x'}{\e_n} \right\rangle, u_n(x')\right)dx' \right)\\
			&= \int_{\Omega} \theta_1(x) \left( \int_Q \int_{\R^d} \varphi(y, \xi) d\nu_{(x,y)}(\xi)dy\right)dx  \int_{\Omega} \theta_2(x') \left( \int_Q \int_{\R^d} \varphi(y', \xi') d\nu_{(x',y')}(\xi')dy'\right) dx'\\
			&= \iint_{\Omega \times \Omega} \theta_1(x)\theta_2(x') \left(\iint_{Q\times Q}\iint_{\R^d\times\R^d} \varphi_1(y, \xi)\varphi_2(y', \xi') d(\nu_{(x,y)} \otimes \nu_{(x',y')})(\xi,\xi')dydy'	\right)dx dx'
		\end{align*}
		
		The fourth equality is due to the boundedness of the sequence $\{ u_n\}$, and the fact that $\left\{\left(\langle \frac{x}{\varepsilon_n}\rangle, u_n(x) \right)\right\}$ generates the two-scale Young measure $\nu$. The thesis follows from the chain of identities. The reverse implication can be proved in a similar way.
	\end{proof}
	Now, we prove Theorem \ref{hom}. This result is a generalization of \cite[Theorem 1.2]{BBS2008}. We observe also that we will refer to the $\Gamma$-convergence as it has been introduced in Definition \ref{defGammafamily}, since by the growth assumption \eqref{pgrowth},  we can work in bounded subsets of $L^p(\Omega;\mathbb R^m)$, which are metrizable.

	\begin{proof}[Proof of Theorem \ref{hom}]
		We recall the family of functionals $I_{\varepsilon}: L^p(\Omega, \R^d) to [0, +\infty)$ defined in \eqref{functI} as
		\begin{equation*}
			I_{\varepsilon}(u):= \int_{\Omega} \int_{\Omega} W\left(x, x', \Big\langle \frac{x}{\varepsilon}\Big\rangle,\Big\langle \frac{x'}{\varepsilon}\Big\rangle,u(x), u(x')\right) dx dx' 
		\end{equation*} 
		where $\Omega \subset \R^N$, $u\in L^p(\Omega;\R^d)$ and $W: \Omega \times \Omega \times Q \times Q \times \R^d \times \R^d \to [0,+\infty)$ is a symmetric admissible integrand, satisfying \eqref{pgrowth}.
		
		\noindent
		Let $u \in L^p(\Omega, \R^d)$ and let $\{\varepsilon_n\} $ be such that $\varepsilon_n \rightarrow 0$. We start by showing that 
		\begin{align}
			\label{gammalimsup}
			&\Gamma-\limsup_{n \rightarrow +\infty}I_{\varepsilon_n}(u)\\
			&\le \inf_{\nu \in \mathcal{M}_u} \iint_{\Omega \times \Omega} \iint_{Q\times Q}\iint_{\R^d \times \R^d} W(x, x',y, y',\xi, \xi')d\nu_{(x,y)}(\xi)d\nu_{(x',y')}(\xi')dydy'dxdx'
		\end{align}
		where $\mathcal{M}_u$ is defined in \eqref{Mu}.  Let $\nu \in \mathcal{M}_u$, by Remark \ref{remarkud}, Theorem \ref{mainresult} and Proposition \ref{Pedregal1997}, there exists a sequence $\{ u_n\} \subset L^p(\Omega, \R^d)$ such that $ u_n \rightharpoonup u $ in $L^p(\Omega, \R^d)$, generates the two-scale Young measure $\{\nu_{(x,y)} \}_{(x,y) \in \Omega\times Q}$, and clearly$\{ (u_n(x), u_n(x') )\}_n$ is the sequence associated to the two-scale Young measure $\{ \nu_{(x,y)}\otimes\nu_{(x',y')} \}_{(x,x', y,y') \in \Omega^2\times Q^2}$. Extract a subsequence $\{ \varepsilon_{n_k} \} \subset \{ \varepsilon_n\}$ such that 
		\[ \limsup_{n \rightarrow \infty} I_{\varepsilon_n}(u_n) = \lim_{k \rightarrow \infty} I_{\varepsilon_{n_k}}(u_{n_k})\]
		and that $\{ |u_{n_k}|^p\}$ is equi-integrable, %for almost every $x,x' \in \Omega$, 
		which is always possible by the Decomposition Lemma \cite[Lemma 8.13]{FLbook}. In particular, due to the $p$-growth condition \eqref{pgrowth}, it follows that the sequence $\{ W(\cdot, \cdot', \langle \cdot/ \varepsilon_{n_k} \rangle, \langle \cdot'/ \varepsilon_{n_k} \rangle, u_{n_k}(\cdot), u_{n_k}(\cdot '))\}$ is equi-integrable in $\Omega \times \Omega$ as well,  hence applying  Theorem \ref{theoremBarchiesi}-$ii)$ we get that
		\begin{align}
			\label{gammalimsup2}
			& \nonumber \Gamma-\limsup_{n \rightarrow +\infty} I_{\varepsilon_n}(u) \le \lim_{k \rightarrow + \infty} \iint_{\Omega \times \Omega} W\left(x, x', \left\langle \frac{x}{\varepsilon_{n_k}}\right\rangle, \left\langle \frac{x'}{\varepsilon_{n_k}}\right\rangle, u_{n_k}(x), u_{n_k}(x')\right)dx dx' \\ &= \iint_{\Omega \times \Omega} \iint_{Q \times Q} \iint_{\R^d \times \R^d} W\left(x, x', y, y', \xi, \xi'\right) d\nu_{(x,y)}(\xi)d\nu_{(x',y')}(\xi')dydy'dxdx'.
		\end{align}
		Taking the infimum over all $\nu \in \mathcal{M}_u$ in the right hand side of \eqref{gammalimsup2}, yields \eqref{gammalimsup}.
		
		Let us prove now that
		\begin{align}
			\label{gammaliminf}
			&\Gamma-\liminf_{n \rightarrow +\infty} I_{\varepsilon_n}(u)\\
			&\ge \inf_{\nu \in \mathcal{M}_u} \iint_{\Omega \times \Omega} \iint_{Q\times Q} \iint_{\R^d \times \R^d} W(x, x', y, y',\xi, \xi') d\nu_{(x,y)}(\xi) d\nu_{(x',y')}(\xi') dy dy' dx dx'. \nonumber
		\end{align} 
		Let $\eta > 0$ and $\{ u_n \} \subset L^p(\Omega, \R^d)$ such that $u_n \rightharpoonup u$ in $L^p(\Omega,\R^d)$ and 
		\begin{equation}
			\label{eq2}
			\liminf_{n \rightarrow +\infty}I_{\varepsilon_n} (u) \le \Gamma-\liminf_{n \rightarrow +\infty}I_{\varepsilon_n}(u) + \eta.
		\end{equation}
		
		For a subsequence $\{ n_k\}$, still thanks to Proposition \ref{Pedregal1997}, we can assume that there exists $\nu \in L^{\infty}_w(\Omega \times Q, \mathcal{M}(\R^d))$ such that $\{ (u_n(x), u_n(x') )\}$ is the sequence associated to the two-scale Young measure $\{ \nu_{(x,y)}\otimes\nu_{(x',y')} \}_{(x,y, x',y') \in (\Omega\times Q)^2}$ and
		\begin{equation}
			\label{eq3}
			\lim_{k \rightarrow +\infty} I_{\varepsilon_{n_k}} (u_{n_k})= \liminf_{n \rightarrow +\infty} I_{\varepsilon_n}(u_n).
		\end{equation}
		We remark that $\{ u_{n_k}\}$ is equi-integrable since it is bounded in $L^p(\Omega, \R^d)$ and $p > 1$. Thus by the Fundamental Theorem on Young Measures (Theorem \ref{theoremYoung}) we get that for every $A \in \mathcal{A}(\Omega)$,
		\[ \int_A u(x) dx = \lim_{k \rightarrow +\infty} \int_A u_{n_k}(x) dx = \int_A \int_Q \int_{\R^d} \xi d\nu_{(x,y)}(\xi)dydx.\]
		By the arbitrariness of the set $A$, it follows that 
		\begin{equation}
			\label{eq1}
			u(x) = \int_Q \int_{\R^d} \xi d\nu_{(x,y)}(\xi)dy \text{ a.e. in } \Omega.
		\end{equation}
		As a consequence of Corollary \ref{corollary}, $\{ \nu_{(x,y) }\}_{(x,y) \in \Omega \times Q}$ is a two-scale Young measure and, by \eqref{eq1}, we also have that $\nu \in \mathcal{M}_u$. Applying now  Theorem \ref{theoremBarchiesi}-$i)$ we get that 
		\begin{align*}
			&\lim_{k \rightarrow +\infty} \iint_{\Omega \times \Omega} W\left( x, x', \left\langle \frac{x}{\varepsilon_{n_k}}\right\rangle, \left\langle \frac{x'}{\varepsilon_{n_k}}\right\rangle, u_{n_k}(x), u_{n_k}(x')\right)dx dx'\\
			&\ge \iint_{\Omega \times \Omega} \iint_{Q \times Q} \iint_{\R^d \times \R^d}  W\left( x, x', y, y', \xi, \xi'\right)d\nu_{(x,y)}(\xi)d\nu_{(x',y')}(\xi')dydy'dxdx' \\
			&\ge \inf_{\nu \in \mathcal{M}_u} \iint_{\Omega \times \Omega} \iint_{Q \times Q} \iint_{\R^d \times \R^d}  W\left( x, x', y, y', \xi, \xi'\right)d\nu_{(x,y)}(\xi)d\nu_{(x',y')}(\xi')dydy'dxdx'.
		\end{align*}
		
		Hence, by \eqref{eq2}, \eqref{eq3} and the arbitrariness of $\eta$ we get the desired result. Gathering \eqref{gammalimsup} and \eqref{gammaliminf} , we obtain that 
		\begin{align*}
			&\Gamma-\lim_{n \rightarrow +\infty} I_{\varepsilon_n}(u)\\
			&= \inf_{\nu \in \mathcal{M}_u} \iint_{\Omega \times \Omega} \iint_{Q \times Q} \iint_{\R^d \times \R^d}  W\left( x, x', y, y', \xi, \xi'\right)d\nu_{(x,y)}(\xi)d\nu_{(x',y')}(\xi')dydy'dxdx'.
		\end{align*}
		
		To conclude the proof of the upper bound, we have to prove that the minimum is attained. Consider a recovering sequence $\{ \bar{u}_n\} \subset L^p(\Omega, \R^d)$ for $\Gamma-\lim_n I_{\varepsilon_n}(u)$.  Arguing exactly as before we can assume that (a subsequence of) $\{ \bar{u}_n\}$ generates a two-scale Young measure $\{ \nu_{(x,y)}\}_{(x,y) \in \Omega \times Q}$, that $\nu \in \mathcal{M}_u$ and the couple $\{(u_n(x), u_n(x')\}_n$ generates $\{\nu_{(x,y)}\otimes \nu_{(x',y')}\}_{x,y,x',y')\in (\Omega \times Q)^2}$, and 
		
		$\{ W(\cdot, \cdot', \langle \cdot/ \varepsilon_{n_k} \rangle, \langle \cdot '/ \varepsilon_{n_k} \rangle, u_{n_k}(\cdot), u_{n_k}(\cdot'))\}$ %, $\{ W(x, \cdot, \langle x/ \varepsilon_{n_k} \rangle, \langle \cdot/ \varepsilon_{n_k} \rangle, u_{n_k}(x), u_{n_k}(\cdot))\}$ are 
		is equi-integrable. According to Theorem \ref{theoremBarchiesi}-ii), we get %and using the fact that $\{\bar{u}_n \}$ is a recovering sequence,
		\begin{align*}
			\Gamma-\lim_{n \rightarrow +\infty} I_{\varepsilon_n}(u) = &\lim_{n \rightarrow +\infty} \iint_{\Omega \times \Omega} W\left(x, x', \left\langle \frac{x}{\varepsilon_n} \right\rangle, \left\langle \frac{x'}{\varepsilon_n} \right\rangle, \bar{u}_n(x), \bar{u}_n(x')\right)dx dx' \\
			=& \iint_{\Omega \times \Omega} \iint_{Q \times Q} \iint_{\R^d \times \R^d}  W\left( x, x', y, y', \xi, \xi'\right)d\nu_{(x,y)}(\xi)d\nu_{(x',y')}(\xi')dydy'dxdx'.
		\end{align*}
		which completes the proof.
	\end{proof}
	
	\begin{remark}
		\label{remBMC}
		\begin{itemize}
			\item[i)] We observe that if in the above theorem $W$ does not depend on $y,y' \in Q$, the above $\Gamma$-convergence result reduces to a relaxation result with respect to the $L^p$-weak convergence, with $I_{\rm hom}$ therein replaced by
			\begin{equation*}
				I_{hom}(u):=\min_{\nu \in \mathcal{M}_u} \iint_{\Omega\times\Omega}\iint_{Q\times Q}\iint_{\R^{d}\times\R^{d}} W(x, x',\xi,\xi')d\nu_{(x,y)}(\xi) d\nu_{(x',y')}(\xi')dydy'dxdx',
			\end{equation*}
			with $\mathcal M_u$ in \eqref{Mu}.
			$I_{\rm hom}$, in turn, becomes
			\begin{equation}\label{Ihom2}
				I_{\rm hom}(u)=\min_{\mu \in \mathcal{M'}_u} \iint_{\Omega\times\Omega}\iint_{\R^{d}\times\R^{d}} W(x, x',\xi,\xi')d\mu_{x}(\xi) d\mu_{x'}(\xi')dxdx',
			\end{equation}
			where 
			\begin{align}
				\label{M'u}
				\nonumber \mathcal{M'}_u := \Big\{ \mu \in L^{\infty}_w(\Omega, \mathcal{M}(\R^d)): \{ \mu_{x}\}_{x\in \Omega}   \text{ is a Young measure such that} \int_{\R^d} \xi d\nu_{x}(\xi) = u(x) \Big\}.
			\end{align}  
			
			\noindent The above equality is easily obtained in view of Definition \ref{def-two-scale}, which ensures that any measure $\{\nu_{(x,y)}\}_{(x,y)\in \Omega \times Q} \in \mathcal M_u$ is such that $\nu_{(x,y)}\otimes dy= \mu_x \in {\mathcal M}'_u$.
			\item[ii)] We observe that the same strategy adopted in the proof of Theorem \ref{hom}, would lead directly to \eqref{Ihom2}, without using two-scale Young measures but adopting the classical Young measures generated by sequence in $L^p(\Omega;\mathbb R^d)$. 
			It suffices to replace Theorem \ref{theoremBarchiesi} by Theorem \ref{theoremYoung}(iv)-(v) and the characterization of Young measures proved in \cite[Theorem 7.7]{Pedregalbook}.
			
			\item[iii)] Observe that equation \ref{Ihom2} provides a statement analogous of \cite[Theorem 6.1]{BMC18} replacing the narrow convergence by the $L^p$-weak convergence, under a slightly more general set of assumptions on $W$, cf. Remarks \ref{remark1}, \ref{remBarchiesi} and \ref{remark2}. 
		\end{itemize}
	\end{remark}
	
	We point out that analogous arguments as those in the proof of Theorem \ref{hom} allow us to obtain a $\Gamma$-convergence result in terms of a suitable narrow convergence (see Remark \ref{narrow}). In order to do so, a preliminary step is needed, i.e. a suitable extension of the functionals $I_\varepsilon$ in \eqref{functI} to $\mathcal Y^p(\Omega;\mathbb T^N\times\mathbb R^d)$, where we recall that $\mathbb T^N$ represents the $N$-dimensional torus in $\mathbb \R^N$.
	
	\begin{corollary}
		\label{cornar}
		Let $p$, $\Omega$, $W$  and $\{I_\varepsilon\}_\varepsilon$ be defined as in Theorem \ref{hom}. Let ${\mathcal Y}^p(\Omega;\mathbb T^N\times \mathbb R^d)$ be the space introduced in Remark \ref{narrow} and let $\bar{I}_\varepsilon:\mathcal{Y}^p(\Omega,\mathbb T^N\times \R^d)\to \R\cup\{\infty\}$
		\begin{equation}\label{Ibarepsilon}
			\bar{I}_\varepsilon(\mu) := \left\{
			\begin{array}{ll} I_\varepsilon(u) &\hbox{ if } \mu= \{\delta_{(\langle\frac{x}{\varepsilon}\rangle, u(x))}\}_{x \in \Omega}, \\
				+\infty &\hbox{ otherwise}
			\end{array}
			\right.
		\end{equation}
		
		Then, $\{{\bar I}_\varepsilon\}_\varepsilon$ $\Gamma$-converges, with respect to the narrow convergence, i.e. testing with functions in $L^1(\Omega;C_0(\mathbb T^N\times \mathbb R^d))$,
		to $\overline{I}_{hom}:\mathcal Y^p(\Omega;\mathbb T^N\times\mathbb R^d) \to [0,+\infty]$, where %${\overline I}_{hom}$ is defined as
		\begin{equation*}
			\overline{I}_{hom}(\mu)=\left\{
			%  \iint_{\Omega \times \Omega}\iint_{(Q\times \mathbb R^d)^2} W(x,x',(y,\xi), (y',\xi'))d \mu_x(y,\xi)d \mu_{x'}(y',\xi') dx dx'
			% \end{equation*}
			\begin{array}{ll}
				\iint_{\Omega \times \Omega}\iint_{(Q\times \mathbb R^d)^2} W(x,x',(y,\xi), (y',\xi'))d \mu_x(y,\xi)d \mu_{x'}(y',\xi') dx dx', 
				\\ \\
				\hspace{5cm}
				\hbox{ if }\{\mu_{x}\}_{x\in \Omega}= \{\nu_{(x,y)}\}_{(x,y)\in \Omega \times Q}\otimes d y,\\
				\\
				+\infty \hbox{ otherwise,}
			\end{array}
			\right.
		\end{equation*}
		with $\nu$ a two-scale Young measure.
	\end{corollary}
	\noindent
	
	Before proving the result we observe that, when $\overline I_{\hom}$ is finite, the right hand side coincides with
	$$
	\iint_{\Omega\times\Omega}\iint_{Q\times Q}\iint_{\R^{d}\times\R^{d}} W(x, x',y,y',\xi,\xi')d\nu_{(x,y)}(\xi) d\nu_{(x',y')}(\xi')dydy'dxdx'.$$

	\begin{proof}
		The proof follows the same argument as \cite[Theorem 6.1]{BMC18}.
		Let $\mu \in {\mathcal Y}^p(\Omega;\mathbb T^N\times \mathbb R^d)$ and $\{\mu_n\} \subset {\mathcal Y}^p(\Omega;\mathbb T^N\times \mathbb R^d)$ be a sequence which converges narrowly to $\mu$. With no loss of generality we can assume that $\liminf_{n\to \infty}{\overline I}_{\varepsilon_n}(\mu_{\varepsilon_n})<+ \infty$. Then, up to passing to a suitable subsequence (not relabelled) for which the liminf is a limit, $\mu_n=\delta_{(\langle\frac{x}{\varepsilon_n}\rangle, u_n(x))}$ for a suitable sequence $\{u_n\}\subset L^p(\Omega;\mathbb R^d)$.
		
		By \eqref{pgrowth}, $\{u_n\}$ is bounded in $L^p(\Omega;\mathbb R^d)$, hence (up to a subsequence) weakly convergent, then as observed in \cite{Barchiesi2006} $\left\{\left(\langle \frac{x}{\varepsilon_n}\rangle, u_n\right)\right\}$ is tight and in view of Theorem \ref{mainresult} the narrow limit $\mu$ is of the type $\nu_{(x,y)}\otimes dy$ for a suitable two-scale Young measure $\{\nu_{(x,y)}\}_{(x,y)\in \Omega \times Q}$.
		Hence the lower semicontinuity of $\iint_{\Omega \times \Omega}\iint_{(Q\times \mathbb R^d)^2} W(x,x',(y,\xi), (y',\xi'))d \mu_x(y,\xi)d \mu_{x'}(y',\xi') dx dx'$ with respect to the narrow convergence, observed in \cite[Proposition 3.7]{BMC18} proves the lower bound. 
		
		For what concerns the upper bound, if $\mu$ is of the type $\nu_{(x,y)}\otimes dy$ for a suitable two-scale Young measure $\{\nu_{(x,y)}\}_{(x,y)\in \Omega \times Q}$, then Theorem \ref{mainresult} ensures the existence of a sequence $\{u_n\}\subset L^p(\Omega;\mathbb R^d)$, weakly convergent in $L^p(\Omega)$ and  such that $\left\{\left(\langle \frac{x}{\varepsilon_n}\rangle, u_n\right)\right\}$ generates the two-scale Young measure $\nu_{(x,y)}$. The Decomposition Lemma ensures that the $\{u_n\}$ can be chosen $p$-equi-integrable. 
		Clearly, by \cite{Barchiesi2006}, the sequence is also narrowly convergent to $\mu_{x}=\nu_{(x,y)}\otimes d y$. Then \eqref{pgrowth} ensures that the integrand of $I_{\varepsilon_n}(u_n)$ is equi-integrable and, exactly as in the proof of Theorem \ref{hom} we can invoke Theorem \ref{theoremBarchiesi}(ii) which, in turn, guarantees the convergence towards 
		$\iint_{\Omega \times \Omega}\iint_{(Q\times \mathbb R^d)^2} W(x,x',(y,\xi), (y',\xi'))d \mu_x(y,\xi)d \mu_{x'}(y',\xi') dx dx'$, 
		which concludes the proof in this case.
		
		Finally if $\mu$ is not of the type $\nu_{(x,y)}\otimes dy$ for a suitable two-scale Young measure $\nu$, then Theorem \ref{mainresult} says that $\{\nu_{(x,y)}\}_{(x,y)\in \Omega \times Q}$ cannot be generated by any $\{\delta_{(\langle\frac{x}{\varepsilon_n}\rangle, u_n(x))}\}$, hence it cannot be the narrow limit of such a sequence.  Consequently, the $\{\mu_n\}$ narrowly converging to $\mu$ is such that the energy $\overline I_{\varepsilon_n}(\mu_n)=+\infty$ for any $n \in \mathbb N$. This concludes the proof.
		
	\end{proof}

	\section*{Acknowledgements}
	The authors gratefully acknowledge support from INdAM GNAMPA.
	A. T. have been partially supported through the INdAM-GNAMPA 2025 project "Minimal surfeces: the Plateau problem and behind" CUP:E5324001950001, and partially supported by the project Geometric-Analytic Methods for PDEs and Applications (GAMPA) - funded by European Union - Next Generation EU within the PRIN 2022 program (D.D. 104 - 02/02/2022 Ministero dell’Università e della Ricerca). This manuscript reflects only the authors’ views and opinions and the Ministry cannot be considered responsible for them.
	G. B. acknowledges the partial support of the INdAM-GNAMPA 2024 Project "Composite materials and microstructures" CUP E53C23001670001.
	E.Z. acknowledges the support of Piano Nazionale di Ripresa e Resilienza (PNRR) - Missione 4 ``Istruzione e Ricerca''
	- Componente C2 Investimento 1.1, ``Fondo per il Programma Nazionale di Ricerca e
	Progetti di Rilevante Interesse Nazionale (PRIN)" - CUP 853D23009360006.  She also acknowledges the partial support of the INdAM - GNAMPA Project 2025 ``Metodi variazionali per problemi dipendenti da operatori frazionari isotropi e anisotropi'' coordinated by Alessandro Carbotti, CUP ES324001950001.
	
	\color{black}

\end{document}